\documentclass{amsart}

\usepackage[a4paper, margin=80pt]{geometry}
\usepackage{amsmath,amsthm,amssymb,bm}
\usepackage{hyperref}
\usepackage{cleveref}
\usepackage{graphicx,color}
\usepackage{tikz}
\usepackage[T1]{fontenc}
\usepackage{tgtermes}
\numberwithin{equation}{section}
\usepackage{enumitem}

\newtheorem{thm}{Theorem}[section]
\newtheorem{lem}[thm]{Lemma}
\newtheorem{prop}[thm]{Proposition}
\newtheorem{cor}[thm]{Corollary}
\theoremstyle{definition}

\newtheorem{defn}[thm]{Definition}

\newtheorem{remark}[thm]{Remark}

\newcommand\Luka{\mathfrak{L}}
\newcommand\gMot{\operatorname{gMot}}
\newcommand\DD{\mathbb{D}}
\newcommand\Sch{\operatorname{Sch}^\circ}
\newcommand\wt{\operatorname{wt}}
\newcommand\wtL{\operatorname{wt}_L}
\newcommand\wtM{\operatorname{wt}_M}
\newcommand\wtS{\operatorname{wt}_S}

\newcommand{\CC}{\mathbb{C}}
\newcommand{\ZZ}{\mathbb{Z}}
\newcommand{\RR}{\mathbb{R}}
\newcommand\LL{\mathcal{L}}
\newcommand\cL{\mathcal{L}}
\newcommand\qand{\quad\mbox{and}\quad}
\newcommand\Qbinom[3]{\genfrac{[}{]}{0pt}{}{#1}{#2}_{#3}}
\newcommand\qbinom[2]{\Qbinom{#1}{#2}{q}}

\definecolor{darkblue}{rgb}{0.0,0,0.7}
\renewcommand\emph[1]{\textcolor{darkblue}{\it #1}}

\title{Combinatorics of orthogonal polynomials on the unit circle}
\author{Jihyeug Jang}
\address{Department of Mathematics, Sungkyunkwan University, Suwon,
  South Korea}
\email{jihyeugjang@gmail.com}
\author{Minho Song}
\address{Department of Mathematics, Sungkyunkwan University, Suwon,
  South Korea}
\email{smh3227@skku.edu}

\begin{document}

\begin{abstract}
  Orthogonal polynomials on the unit circle (OPUC for short) are a
  family of polynomials whose orthogonality is given by integration
  over the unit circle in the complex plane. There are combinatorial
  studies on the moments of various types of orthogonal polynomials,
  including standard orthogonal polynomials, Laurent biorthogonal
  polynomials, and orthogonal polynomials of type \( R_I \). In this
  paper, we study the moments of OPUC from a combinatorial
  perspective. We provide three path interpretations for them:
  \L{}ukasiewicz paths, gentle Motzkin paths, and Schr\"oder paths.
  Additionally, using these combinatorial interpretations, we derive
  explicit formulas for the generalized moments of some examples of
  OPUC, including the circular Jacobi polynomials and the
  Rogers--Szeg\H{o} polynomials. Furthermore, we introduce several
  kinds of generalized linearization coefficients and give
  combinatorial interpretations for them.
\end{abstract}

\maketitle

\section{Introduction}\label{sec:intro}

Orthogonal polynomials originated from the study of continued
fractions and have become significant objects of study in various
fields, including classical analysis, partial differential equations,
mathematical physics, probability, random matrix theory, and
combinatorics. A polynomial with real coefficients in the variable
\( x \) is called \emph{monic} if its leading coefficient is \( 1 \).
For a positive Borel measure \( \mu \) on \( \RR \) with infinite
support, a polynomial sequence \( ( P_n(x) )_{n\geq 0} \) is called a
sequence of \emph{(monic) standard orthogonal polynomials} with
respect to \( \mu \) if for all \( n,m \geq 0 \),
\( \deg P_n(x) = n \) and
\[
  \int_{\RR} P_m(x)P_n(x) d\mu(x) = K_n \delta_{m,n}, \quad K_n > 0.
\]
The orthogonality condition can be expressed in terms of a linear functional \( \cL \):
\[
  \cL(P_m(x)P_n(x)) = K_n \delta_{m,n}, \quad K_n > 0.
\]
According to Favard's theorem \cite[Theorem~4.4]{Chihara1978} and
\cite[Theorem~2.5.2]{Ismail2009}, a sequence
\( ( P_n(x) )_{n \geq 0} \) is a sequence of monic orthogonal
polynomials with respect to a measure \( \mu \) with
\( \int_{\RR}d\mu(x) = 1 \) (equivalently, a unique linear functional
\( \cL \) with \( \cL(1) = 1 \)) if and only if the polynomials
\( P_n(x) \) satisfy a three-term recurrence relation
\[
  P_{n+1}(x) = (x-b_n)P_n(x) - \lambda_nP_{n-1}(x) \quad\mbox{for}\quad n\geq 0,
\]
with the initial conditions \( P_{-1}(x) = 0 \) and \( P_0(x) = 1 \)
for some sequences \( \bm b = (b_n)_{n\geq 0} \) and
\( \bm\lambda = (\lambda_n)_{n \geq 1} \) with
\( \lambda_n > 0 \).

The \emph{moment} \( \mu_n \) of standard orthogonal polynomials is
defined by \( \mu_n = \int_{\RR}x^nd\mu(x) = \cL(x^n) \). Viennot
\cite{Viennot1983, Viennot1985} found a combinatorial interpretation
for the moment of standard orthogonal polynomials:
\[ \mu_n = \sum_{p \in \operatorname{Mot}_n}\wt(p),
\] where \( \operatorname{Mot}_n \) is the set of Motzkin paths \( p
\) from \( (0,0) \) to \( (n,0) \) and \( \wt(p) \) is a certain
weight of \( p \) depending on the sequences \( \bm b \) and \(
\bm\lambda \). Following this research, there has been active
exploration into combinatorial interpretations of moments for
well-known examples of standard orthogonal polynomials. For instance,
the moments of Hermite, Charlier, and Laguerre polynomials are
expressed as generating functions for perfect matchings, set
partitions, and permutations, respectively. See the survey paper
\cite{CKS2016} for more details.

There are combinatorial studies of various types of orthogonal
polynomials, not just for the standard ones. For instance, Kamioka
\cite{Kamioka2007} showed that the moment \( \mu_n \) of the Laurent
biorthogonal polynomials is equal to the weight sum of Schr\"oder
paths of length \( 2n \). In addition, Kim and Stanton \cite{KS2023}
showed that the moment \( \mu_n \) of the orthogonal polynomials of
type \( R_I \) is equal to the weight sum of Motzkin--Schr\"oder paths
of length \( n \). Inspired by the previous works, in this paper, we
investigate combinatorial aspects of orthogonal polynomials on the
unit circle. One of our goals is to find a combinatorial model
corresponding to the `???' part below.
\begin{align*}
 \mbox{Standard orthogonal polynomials} &\to \mbox{Motzkin paths \cite{Viennot1983, Viennot1985}} \\
  \mbox{Laurent biorthogonal polynomials} &\to \mbox{Schr\"oder paths \cite{Kamioka2007}} \\
  \mbox{Orthogonal polynomials of type \( R_I \)} &\to \mbox{Motzkin--Schr\"oder paths \cite{KS2023}} \\
  \mbox{Orthogonal polynomials on the unit circle} &\to \mbox{???}
\end{align*}

Let us first review the basic concepts of the orthogonal polynomials
on the unit circle. For more details, refer to \cite{Simon2005a,
  Simon2005, Ismail2009}.

Orthogonal polynomials on the unit circle are a family of polynomials
whose orthogonality is defined through integration over the unit
circle in the complex plane. Precisely, for a given nontrivial
probability measure \( \mu(\theta) \) on \( [-\pi,\pi] \), a
polynomial sequence \( ( \Phi_n(z) )_{n \geq 0} \) is called a
sequence of \emph{(monic) orthogonal polynomials on the unit circle
  (OPUC for short)} with respect to \( \mu(\theta) \) if
\( \deg \Phi_n(z) = n \), \( [z^n]\Phi_n(z) = 1 \), and it satisfies
the orthogonality condition
\[
  \int_{[-\pi,\pi]} \Phi_m(\zeta) \overline{\Phi_n(\zeta)} d\mu(\theta) = \kappa_n \delta_{m,n} \quad \mbox{for}\quad \zeta = e^{i\theta} \mbox{ and } \kappa_n > 0.
\]

For a polynomial \( f(z)=\sum_{k=0}^na_kz^k \) of degree \( n \) with
\(a_n \neq 0 \), let
\( \overline{f}(z)=\sum_{k=0}^n\overline{a_k}z^k \). The \emph{reverse
  polynomial} \( f^* \) is defined as \( z^n \overline{f}(1/z) \), or
equivalently, \( f^*(z) = \sum_{k=0}^n \overline{a_k} z^{n-k}\).

Analogous to Favard's theorem for standard orthogonal polynomials,
Verblunsky's theorem states that \( ( \Phi_n(z) )_{n\geq 0} \) is a
sequence of monic orthogonal polynomials on the unit circle with
respect to a unique measure \( \mu \) with
\( \int_{[-\pi,\pi]}d\mu(\theta) =1 \) if and only if it satisfies the
following recurrence relations known as Szeg\H{o}'s recurrence
relations:
\begin{align}
  \label{eq:rec1}\Phi_{n+1}(z) &= z\Phi_n(z) - \overline{\alpha_n}\Phi_n^*(z), \text{ or equivalently, }\\
  \label{eq:rec2}\Phi_{n+1}^*(z) &= \Phi_n^*(z) -\alpha_nz\Phi_n(z),
\end{align}
for some complex sequence \( (\alpha_n)_{n\ge0} \) with
\( |\alpha_n| <1 \). The "only if" part can be easily proven, and
there are surprisingly five different proofs for its converse in
\cite{Simon2005a}.

The \emph{moments} of OPUC are defined by
\( \mu_n= \int_{[-\pi,\pi]} e^{-in\theta}d\mu(\theta) \) for integers
\( n \). Using Szeg\H{o}'s recurrence relations and the orthogonality
condition, the moments can be obtained recursively. For instance, the
moments \( \mu_n \) for \( n=1,2,3 \) are
\begin{align}
 \notag \mu_1 &= \alpha_0, \\
  \notag \mu_2 &= \alpha_0^2 + \alpha_1(1-|\alpha_0|^2), \\
  \label{eq:mu3} \mu_3 &= \alpha_0^3 + 2\alpha_0\alpha_1(1-|\alpha_0|^2) - \alpha_1^2 \overline{\alpha_0}(1-|\alpha_0|^2) + \alpha_2(1-|\alpha_0|^2)(1-|\alpha_1|^2).
\end{align}

From a combinatorial perspective, there are a few results worth
noting. One of them is Verblunsky's formula \cite[(i) and (ii) in
Theorem~8.2.2]{Ismail2009}, which says polynomialities of coefficients
of \( \Phi_n(z) \) and \( \mu_n \) in terms of
\( \alpha_0, \overline{\alpha_0}, \alpha_1, \overline{\alpha_1},
\cdots \). Additionally, Golinskii and Zlatos
\cite[Theorem~1.1]{GZ2007} gave an explicit formula for the
coefficients of \( \Phi_n(z) \), which strengthens the first item of
Verblunsky's formula. In this context, a question of combinatorial
interpretations for the moments naturally arises.

To derive combinatorial interpretations of the moments of OPUC, we use
the recurrence relations of OPUC, similar to the approach used for
other types of orthogonal polynomials. Instead of interpreting
\(\mu_n\) directly, we introduce new parameters \(r\) and \(s\) to
define a more general moment \(\mu_{n,r,s}\) so that the moment
\(\mu_n\) follows as the special case \( r=s=0 \). To achieve this, we
translate the definitions written in terms of integrals into linear
functionals and explain the orthogonality property using an inner
product.

For a given measure \( \mu \), define a linear functional \( \LL \) by
\( \LL(f(z))= \int_{-\pi}^{\pi} f(e^{i \theta}) d\mu(\theta) \). Then
it follows \( \mu_n = \LL(z^{-n}) \). For an integer \( n \) and
nonnegative integers \( r \) and \( s \), we define a
\emph{generalized moment} \( \mu_{n,r,s} \) by
\[
  \mu_{n,r,s} = \frac{\cL(z^{-n}\cdot\overline{\Phi_r(z)}\cdot\Phi_s(z))}{\cL(\Phi_s(z)\cdot\overline{\Phi_s(z)})}.
\]
It follows directly from the definition that
\( \mu_{n,0,0} = \mu_n \). The concept of \( \mu_{n,r,s} \) natually
arises as a generalization in the combinatorial study of orthogonal
polynomials, as it serves as the coefficient in the expansion of
\( z^n \Phi_r(z) \) in terms of the basis \( (\Phi_s)_{s\geq 0} \).
Specifically, we have
\[
  z^n\Phi_r(z) = \sum_{s\geq 0} \overline{\mu_{n,r,s}} \cdot \Phi_s(z).
\]
Further details on linearization will be discussed in
\Cref{sec:generalized_lc}.

We provide three lattice path interpretations and obtain the
reciprocity for the generalized moment of OPUC. See~\Cref{sec:Moments
  and path} for their precise definitions.
\begin{thm}
  For nonnegative integers \( n \), \( r \), and \( s \), the
  generalized moment of OPUC can be expressed as
\[
  \mu_{n,r,s} = \sum_{p \in \mathfrak{L}_{n,r,s}}\wt_L(p) = \sum_{p \in \operatorname{gMot}_{n,r,s}}\wt_M(p),
\]
where \( \mathfrak{L}_{n,r,s} \) is the set of \emph{\L{}ukasiewicz
  paths} from \( (0,r) \) to \( (n,s) \) and
\( \operatorname{gMot}_{n,r,s} \) is the set of \emph{gentle Motzkin
  paths} from \( (-r,r) \) to \( (2n-s,s) \) for certain weight
functions \( \wt_L \) and \( \wt_M \). In particular, with the
assumption that \( \alpha_i \not= 0 \) for all \( i \), the
generalized moment can also be expressed as
\[
  \mu_{n,r,s} = \sum_{p \in \Sch_{n,r,s}}\wt_S(p),
\]
where \( \Sch_{n,r,s} \) is the set of \emph{Schr\"oder paths} from
\( (0,r) \) to \( (n,s) \) that do not start with a vertical down-step
\( (0,-1) \) for a certain weight function \( \wt_S \).

Moreover, for a negative integer \( -n \), we have
\[
  \mu_{-n,r,s} = \overline{\mu_{n,s,r}}\cdot \frac{\prod_{j=0}^{s-1}(1-|\alpha_j|^2)}{\prod_{j=0}^{r-1}(1-|\alpha_j|^2)}.
\]
\end{thm}

These results are proved in \Cref{thm:Luka}, \Cref{thm:recipociry},
\Cref{thm:gentle Motzkin}, \Cref{thm:Sch}.

As consequences of combinatorial formulas for the moments, we not only
show the polynomiality of the generalized moments but also provide a
condition for their positivity, as stated in \Cref{thm:generalized
  Verblunsky}. We also explain relations between the three models
combinatorially. Additionally, using these combinatorial
interpretations, we derive explicit formulas for the generalized
moments of several examples of OPUC, including the circular Jacobi
polynomials and the Rogers--Szeg\H{o} polynomials.

This paper is organized as follows. In~\Cref{sec:lf}, we re-examine
basic properties of OPUC in terms of a linear functional.
In~\Cref{sec:Moments and path}, we provide three combinatorial models
for the moments: \L{}ukasiewicz paths, gentle Motzkin paths, and
Schr\"oder paths. We also show relations between them in a
combinatorial way. In~\Cref{sec:examples}, we provide explicit
formulas for generalized moments of several examples of OPUC using the
lattice path models. In~\Cref{sec:generalized_lc}, we define four
kinds of generalized linearization coefficients and interpret them
combinatorially.

\section{OPUC with respect to a linear functional \( \cL \)}\label{sec:lf}

For the combinatorics of OPUC, it is not necessary to know all their
analytic background. In this section, we revisit some basic properties
of OPUC needed from a combinatorial perspective, describing them using
only a linear functional, not an integral. There are some cautions
when writing everything in terms of a linear functional. For example,
for the OPUC with respect to the measure introduced
in~\Cref{sec:intro}, the property \( \mu_n = \overline{\mu_{-n}} \)
naturally follows from the use of an integral. But it is not so
naturally obtained when we develop the theory of OPUC in terms of a
linear functional. This will be explained in~\Cref{lem:conjugate_mu}.
In \Cref{lem:equiv_ortho_cond}, we provide an equivalent condition for
the orthogonality condition~\eqref{eq:ortho_cond} introduced below.
Moreover, \Cref{lem:conjugate_mu} and~\Cref{lem:equiv_ortho_cond} have
an advantage that they clarify the differences between OPUC and
Laurent biorthogonal polynomials.

One more interesting result is about the determinants. \Cref{prop:
  det=det} explains a relation between the determinant of the Hankel
matrix \( (\mu_{m+i-j})_{0\le i,j\le n} \) and that of
\( (\mu_{m+i,j}) \), leading to the positive definiteness of
\( (\mu_{i-j})_{0\le i,j\le n} \). 

Although this is not a survey paper, we believe this section is
helpful for those who are not familiar with the analytic background of
OPUC but are interested in the combinatorial aspects of OPUC.

\begin{defn}\label{def:OPUC}
  A sequence \( (\Phi_n(z))_{n\ge0} \) of monic polynomials with
  \( \deg\Phi_n(z) = n \) is called a sequence of \emph{(monic)
    orthogonal polynomials on the unit circle (OPUC for short)} with
  respect to \( \cL \) if it satisfies the orthogonality condition
  given in terms of a certain linear functional \( \cL \) on \( V \):
\begin{align}
\label{eq:ortho_cond}  \cL(\Phi_m(z)\overline{\Phi_n}(1/z)) = \kappa_n \delta_{m,n}, \quad \kappa_n > 0,
\end{align}
where \( V \) is the vector space of Laurent polynomials given by
\[
  V=\left\{\sum_{k \in \ZZ}f_kz^k : f_k\in\CC, f_k\neq0 \text{ for only finitely many } k\in\ZZ \right\}.
\]
\end{defn}
Note that the same linear functional \( \cL \) is given in
\cite[(1.3.13)]{Simon2005a}, where its well-definedness is also
explained. In the literature, it is written as
\( \cL(\Phi_m(z) \overline{\Phi}_n(1/z)) =
\kappa_n\delta_{n,m} \) for \( \kappa_n >0 \), which is equivalent to
our definition since \( 1/\overline{z} = z \) for \( z \) on the unit
circle \( \partial \DD \) in the complex plane.

To define a linear functional \( \cL \), it is enough to determine the values of \( \cL \) on a basis for \( V \).
For any given monic polynomial sequence \( (\Phi_n(z))_{n\ge0} \) with \( \deg\Phi_n(z) = n \), two bases for \( V \) are
\begin{enumerate}
\item \( \{z^n : n \in \ZZ \} \), and
\item \( \{\Phi_n(z):n\ge0\}\cup \{\overline{\Phi_n}(1/z):n>0 \} \).
\end{enumerate}
 The following lemma provides a condition for \( \mu_n = \overline{\mu_{-n}} \), which we call the reciprocity of the moments, in terms of the values of \( \cL \) on a basis.

\begin{lem}\label{lem:conjugate_mu}
  For a given monic polynomial sequence \( (\Phi_n(z))_{n\ge0} \) with
  \( \deg\Phi_n(z) = n \), let \( \cL \) be the linear functional on
  \( V \) defined by \( \cL(1) = 1 \) and
  \( \cL(\Phi_n(z)) = \cL(\overline{\Phi_n}(1/z)) = 0 \) for positive
  integers \( n \). Then, for \( \mu_n = \cL(z^{-n}) \), we have
  \( \mu_n = \overline{\mu_{-n}} \) for positive integers \( n \).
\end{lem}
\begin{proof}
 Let \( \Phi_n(z)=\sum_{k=0}^n c_{n,k}z^k \) with \( c_{n,n} = 1 \).
Since \( \cL(\Phi_0(z))=\cL(1)=1 \) and \( \cL(\Phi_n(z))=0 \) for \(
n\ge1 \), we have
   \[ (c_{i,j})_{i,j\ge0}(\mu_{-i})_{i\ge0}=(1,0,0,\dots)^T.
   \] The matrix \( (c_{i,j})_{i,j\ge0} \) is a lower triangular
matrix with all 1's on the main diagonal and hence it is invertible.
Let \( (c_{i,j})_{i,j\ge0}^{-1}=(d_{i,j})_{i,j\ge0} \). Then \(
\mu_{-i} = d_{i,0} \). Similarily, since \(
\cL(\overline{\Phi_0}(1/z))=\cL(1)=1 \) and \(
\cL(\overline{\Phi_n}(1/z))=0 \) for \( n\geq 1 \), we have
   \[ (\overline{c_{i,j}})_{i,j\ge0}(\mu_{i})_{i\ge0}=(1,0,0,\dots)^T,
   \] and hence \( \mu_i =\overline{d_{i,0}} = \overline{\mu_{-i}} \),
as required.
\end{proof}

Next, we provide an equivalent condition for the orthogonality
condition~\eqref{eq:ortho_cond}.

\begin{lem}\label{lem:equiv_ortho_cond}
  Let \( ( \Phi_n(z) )_{n\geq 0} \) be an arbitrary sequence of monic
  polynomials with \( \deg\Phi_n(z)=n \). A linear functional
  \( \cL \) on \( V \) satisfies
  \begin{equation}\label{eq:lem_ortho}
    \cL(\Phi_m(z)\overline{\Phi_n}(1/z)) = \kappa_n \delta_{m,n}
  \end{equation}
  for some (possibly zero) constant \( \kappa_n \), and for all
  nonnegative integers \( n \) and \( m \) if and only if it satisfies
  \begin{align}
\label{eq:equiv_cond1}  \cL(\overline{\Phi_n}(1/z)) &= 0, \qquad\quad \;\; n\ge1,\\
 \label{eq:equiv_cond2} \cL(\Phi_m(z)z^{-n}) &= \kappa_n\delta_{m,n}, \quad  0 \le n \le m.
\end{align}
 
\end{lem}
\begin{proof}
  Consider a linear functional \( \cL \)
  satisfying~\eqref{eq:lem_ortho}. We obtain~\eqref{eq:equiv_cond1}
  immediately from~\eqref{eq:lem_ortho} by letting \( m=0 \). To
  show~\eqref{eq:equiv_cond2}, we use induction on \( n \). The base
  case \( n = 0 \) follows from~\eqref{eq:lem_ortho}. Suppose that the
  equation~\eqref{eq:equiv_cond2} holds for \( n = 0,\dots, N-1 \).
  Since \( \Phi_N(z) \) is monic, we can write
  \( \overline{\Phi_N}(1/z) = z^{-N} + \sum_{n=0}^{N-1}c_nz^{-n} \)
  for some \( c_n \). Thus, for \( N \) and \( m \) with
  \( m\ge N\ge 1 \), we have
  \[
    \kappa_N \delta_{m,N} = \cL(\Phi_m(z)\overline{\Phi_N}(1/z)) =\cL(\Phi_m(z)z^{-N}) +\sum_{n=0}^{N-1}c_n\cL(\Phi_m(z)z^{-n}) = \cL(\Phi_m(z)z^{-N}),
  \]
  as required.

  To prove the converse, consider a linear functional \( \cL \)
  satisfying \eqref{eq:equiv_cond1} and~\eqref{eq:equiv_cond2}. For
  \( n \le m \), by the linearity of \( \cL \), we have
  \begin{align}
    \label{eq:equiv_pf}    \cL(\Phi_m(z)\overline{\Phi_n}(1/z)) = \cL(\Phi_m(z)z^{-n}) = \kappa_n\delta_{n,m}.
  \end{align}
  Suppose \( n > m \). By \Cref{lem:conjugate_mu} and the linearity of
  \( \cL \), we have
  \( \cL(f(z)) = \overline{\cL(\overline{f}(1/z))} \) for any \( f \)
  in \( V \). Thus,
  \[
    \cL(\Phi_m(z)\overline{\Phi_n}(1/z)) = \overline{\cL(\Phi_n(z)\overline{\Phi_m}(1/z))}  = 0,
  \]
  where we use \eqref{eq:equiv_pf} for the last equality. This
  completes the proof.
\end{proof}

\begin{remark}\label{rmk:LBP_OPUC}
  In this remark, we clarify potentially confusing aspects
  of~\Cref{lem:equiv_ortho_cond} and describe the differences between
  OPUC and Laurent biorthogonal polynomials. For a given sequence
  \( (\Phi_n(z))_{n\ge0} \) of OPUC with respect to \( \cL \), it is
  known \cite{Simon2005a} that \eqref{eq:ortho_cond} is equivalent
  to~\eqref{eq:equiv_cond2}. Let us make it clear that this does not
  contradict~\Cref{lem:equiv_ortho_cond}. \Cref{lem:equiv_ortho_cond}
  presents conditions~\eqref{eq:equiv_cond1}
  and~\eqref{eq:equiv_cond2} as an equivalent condition for an
  "arbitrary" polynomial sequence satisfying (2.1), not specifically
  for OPUC.

  In fact, a polynomial sequence defined only
  by~\eqref{eq:equiv_cond2} is a sequence of Laurent biorthogonal
  polynomials, see \cite{Zhedanov1998}.
  Rephrasing~\Cref{lem:equiv_ortho_cond} using Laurent biorthogonal
  polynomials, it means that OPUC are the special case of Laurent
  biorthogonal polynomials that they additionally
  satisfy~\eqref{eq:equiv_cond1}. Equivalently, together with
  \Cref{lem:conjugate_mu}, OPUC are considered as the special case of
  Laurent biorthogonal polynomials with the additional condition
  \( \mu_n = \overline{\mu_{-n}} \).
\end{remark}

One can prove the following linear functional version of Verblunsky's
theorem using~\Cref{lem:conjugate_mu} and~\Cref{lem:equiv_ortho_cond}.
However, since the proof follows essentially the same argument as in
\cite{Simon2005a}, we omit the details.

\begin{thm}[The linear functional version of Verblunsky's theorem]\label{thm:linear functional}
  Let \( ( \Phi_n(z) )_{n\geq 0} \) be a sequence of monic polynomials
  with \( \deg\Phi_n(z)=n \) and \( \Phi_0(z) = 1 \). Then it
  satisfies Szeg\H{o}'s recurrence relations:
\begin{align*}
  \Phi_{n+1}(z) &= z\Phi_n(z) - \overline{\alpha_n}\Phi_n^*(z), \text{ or equivalently, }\\
  \Phi_{n+1}^*(z) &= \Phi_n^*(z) -\alpha_nz\Phi_n(z)
\end{align*}
for some complex sequence \( \left( \alpha_n \right)_{n \geq 0} \)
with \( |\alpha_n| < 1 \) if and only if there is a unique linear
functional \( \cL \) on \( V \) satisfying \( \cL(1)= \kappa_0 = 1 \)
together with the orthogonality condition \eqref{eq:ortho_cond}:
\begin{equation*}
  \cL(\Phi_m(z)\overline{\Phi_n}(1/z)) = \kappa_n \delta_{m,n}, \quad \mbox{for } \kappa_n > 0.
\end{equation*}
\end{thm}

Note that one can see that
\[
 \Phi_n(0) = -\overline{\alpha_{n-1}} \quad\mbox{and}\quad \kappa_n = (1-|\alpha_{n-1}|^2)\kappa_{n-1} = \cdots = \prod_{i=0}^{n-1}(1-|\alpha_i|^2).
\]
By convention, we set \( \alpha_{-1} = -1 \).

For a given linear functional \( \cL \) on the vector space \( V \) of
Laurent polynomials, define a map
\( \langle \cdot,\cdot \rangle_\cL: V\times V \to \CC \) by
\begin{align}
\label{eq:inner_product}  \langle f,g \rangle_{\cL}=\cL(f(z)\overline{g}(1/z)).
\end{align}
We write \( f \) and \( g \) instead of \( f(z) \) and \( g(z) \)
inside \( \langle \cdot,\cdot \rangle_\cL \), unless we specifically
consider expressions such as \( f(1/z) \). For brevity, we omit the
subscript \( \cL \) if it is clear from the context.

The following proposition shows that the determinant of a shifted
version of the Toeplitz matrix \( (\mu_{i-j})_{i,j \geq 0} \) is equal
to the determinant of the matrix related to the generalized moments
with a certain factor. Using this, the determinant of the Toeplitz
matrix can be easily obtained as shown in \Cref{cor:det}.

\begin{prop}\label{prop: det=det}
  For an integer \( m \) and a nonnegative integer \( n \), we have
  \begin{equation}
    \det\left( \mu_{m+i-j} \right)_{0\le i,j\le n} = \prod_{k=0}^{n-1}(1-|\alpha_k|^2)^{n-k} \cdot \det\left( \mu_{m+i,j} \right)_{0\le i,j\le n},
  \end{equation}
  where \( \mu_{i,j} =  \langle \Phi_i,z^j \rangle/\langle \Phi_i, \Phi_i \rangle \).
  We set \( \prod_{k=0}^{n-1}(1-|\alpha_k|^2)^{n-k} = 1 \) when \( n=0 \).
\end{prop}
\begin{proof}
  Let \( z_i \) be the variables and \( V_{z_i} \) be the vector space
  of Laurent polynomials in variable \( z_i. \) Set \( \cL_{z_i} \) to
  be the linear functional only acts on \( V_{z_i} \) that satisfies
  the orthogonality condition. Then
  \( \langle \Phi_j,\Phi_j \rangle_{\cL_{z_i}} =
  \prod_{k=0}^{j-1}(1-|\alpha_k|^2) \) for all \( i \) and \( j \).
  When we take \( z_i \) as the variable considered in
  \( \mu_{m+i,j} \), namely,
  \[
    \mu_{m+i,j}=\cL_{z_i}(z_i^{-(m+i)}\Phi_j(z_i))/ \langle \Phi_j,\Phi_j \rangle_{\cL_{z_i}},
  \]
  the following computation gives the result:
  \begin{align*}
    \det\left( \mu_{m+i,j} \right)_{0\le i,j\le n}
    &= \det\left( \cL_{z_i}(z_i^{-(m+i)}\Phi_j(z_i))/ \langle \Phi_j,\Phi_j \rangle_{\cL_{z_i}}  \right)_{0\le i,j\le n}\\
    &= \prod_{k=0}^{n-1}(1-|\alpha_k|^2)^{-(n-k)}\cL_{z_n}\circ\dots\circ\cL_{z_0}\left( \det\left( z_i^{-(m+i)}\Phi_j(z_i) \right)_{0\le i,j\le n}\right)  \\ 
    &=\prod_{k=0}^{n-1}(1-|\alpha_k|^2)^{-(n-k)}\cL_{z_n}\circ\dots\circ\cL_{z_0}\left( \det\left( z_i^{-(m+i-j)} \right)_{0\le i,j\le n}\right)  \\
    &= \prod_{k=0}^{n-1}(1-|\alpha_k|^2)^{-(n-k)}\det\left( \cL_{z_i}(z_i^{-(m+i-j)}) \right)_{0\le i,j\le n} \\
    &= \prod_{k=0}^{n-1}(1-|\alpha_k|^2)^{-(n-k)}\det\left( \mu_{m+i-j} \right)_{0\le i,j\le n}.\qedhere 
  \end{align*}
\end{proof}

\begin{cor}\label{cor:det}
  For a nonnegative integer \( n \), the determinant of the Toeplitz
  matrix \( \left( \mu_{i-j} \right)_{0\le i,j\le n} \) is given by
  \[
    \det\left( \mu_{i-j} \right)_{0\le i,j\le n}  = \prod_{k=0}^{n-1}(1-|\alpha_k|^2)^{n-k}.
  \]
\end{cor}
\begin{proof}
  By~\Cref{lem:equiv_ortho_cond}, the matrix
  \( \left( \mu_{i,j} \right)_{0 \leq i,j \leq n} \) is a lower
  triangular matrix with diagonal entries \( \mu_{i,i} = 1 \) for all
  \( i \), so we have
  \[
    \det\left( \mu_{i,j} \right)_{0\le i,j\le n} = 1.
  \]
  Plugging \( m = 0 \) into \Cref{prop: det=det} completes the proof.
\end{proof}

\begin{remark}
  From~\eqref{eq:ortho_cond}, the map~\eqref{eq:inner_product} is positive-definite and hence defines an inner product.  
More generally, in \eqref{eq:ortho_cond} one may allow \( \kappa_n \neq 0 \) instead of \( \kappa_n > 0 \).  
When \( \kappa_n > 0 \), the linear functional corresponds to a nontrivial probability measure.  
Moreover, while \( \kappa_n > 0 \) implies \( |\alpha_n| < 1 \), the weaker condition \( \kappa_n \neq 0 \) only implies \( |\alpha_n| \neq 1 \).  
In this case, the map \( \langle \cdot,\cdot \rangle \) is no longer positive-definite and thus not an inner product, but this does not affect the combinatorial results established later.  
A detailed discussion of this case can be found in~\cite{CMMV2016,Geronimus1954}.
\end{remark}

\section{Combinatorial interpretations for the generalized moments}\label{sec:Moments and path}

In this section, we provide three combinatorial interpretations for
the generalized moment \( \mu_{n,r,s} \) defined below. Using these
combinatorial interpretations, we extend the second item of
Verblunsky's formula \cite[(ii) in Theorem~8.2.2]{Ismail2009} to
results for generalized moments as follows. This extension follows
directly from the results of \Cref{thm:Luka} (or \Cref{thm:gentle
  Motzkin}) and \Cref{thm:recipociry}.

\begin{defn}\label{def:mu_nrs}
  For an integer \( n \), and nonnegative integers \( r \) and
  \( s \), we define a \emph{generalized moment} \( \mu_{n,r,s} \) for
  OPUC by
\begin{align*}
  \mu_{n,r,s} = \frac{\langle \Phi_s,z^n \Phi_r \rangle}{\langle \Phi_s,\Phi_s \rangle}.
\end{align*}
\end{defn}
Note that
\( \mu_n = \mathcal{L}(z^{-n}) = \langle 1,z^n \rangle =
\mu_{n,0,0}\), \( \mu_{n,m} = \mu_{n,0,m} \) and
\( \langle \Phi_m,\Phi_m \rangle = \prod_{j=0}^{m-1}(1-|\alpha_j|^2)
\).

\begin{thm}[A generalized version of Verblunsky's formula]\label{thm:generalized Verblunsky}
  For nonnegative integers \( n \), \( r \) and \( s \), the
  generalized moment \( \mu_{n,r,s} \) is a polynomial in
  \( \{\alpha_j: 0\leq j < n+r \} \) and
  \( \{ \overline{\alpha_j} : 0 \leq j < n+r \} \) with integer
  coefficients. In particular, if we write
  \( \beta_j = -\overline{\alpha_j} \) for all \( j \), then
  \( \mu_{n,r,s} \) has all positive coefficients in
  \( \ZZ[\alpha_0,\alpha_1,\dots,\beta_0,\beta_1,\dots] \). Moreover,
  we have the reciprocity
  \[
    \mu_{-n,r,s} = \overline{\mu_{n,s,r}}\cdot \frac{\prod_{j=0}^{s-1}(1-|\alpha_j|^2)}{\prod_{j=0}^{r-1}(1-|\alpha_j|^2)},
  \]
  which implies \( \mu_{-n} = \overline{\mu_n} \).
\end{thm}

The generalized moment \( \mu_{n,r,s} \) is related to the
linearization coefficient. Here, the linearization coefficients refer
to the coefficients when a product of polynomials is expressed in
terms of the basis \( \{\Phi_k(z)\}_{k \geq 0} \). Denote by
\( a_{n,r,k} \) the coefficient of \( \Phi_k(z) \) when the polynomial
\( z^n \Phi_r(z) \) is expanded in terms of
\( \{\Phi_k(z)\}_{k \geq 0} \), that is,
\( z^n \Phi_r(z) = \sum_{k \geq 0} a_{n,r,k} \Phi_k(z) \). By
orthogonality, we have
\( \langle \Phi_s, z^n \Phi_r \rangle = \overline{a_{n,r,s}} \cdot
\langle \Phi_s, \Phi_s \rangle \). In other words,
\( \mu_{n,r,s} = \overline{a_{n,r,s}} \). Various types of
linearization coefficients can be found in \Cref{sec:generalized_lc}.

We first review the notations for lattice paths. A \emph{lattice path}
is a finite sequence \( p = (p_0, p_1,\dots,p_n) \) of points in
\( \ZZ \times \ZZ_{\geq 0} \), that is, \( p \) lies weakly above the
\( x \)-axis. Each \emph{step} in the lattice path is defined by the
difference \( (x_i - x_{i-1}, y_i - y_{i-1}) \) between consecutive
points \( p_i = (x_i,y_i) \) and \( p_{i-1} = (x_{i-1},y_{i-1}) \), or
simply denoted as \( p_{i-1} \to p_i \). For convenience, we call the
step \( (1,1) \) an \emph{up-step}, \( (1,0) \) a \emph{horizontal
  step}, \( (0,-1) \) a \emph{vertical down-step}, and \( (1,-\ell) \)
an \emph{\( \ell \)-down-step}, respectively. We simply call the
\( 1 \)-down-step a \emph{down-step}.

In many cases, we consider lattice paths with a specified set of
steps. A \emph{lattice path with the step set \( S \)} is a lattice
path where every step is an element of \( S \). A step set consists of
all possible steps in the lattice path. For example, a \emph{Motzkin
  path} is a lattice path whose steps are an up-step \( (1,1) \), a
horizontal step \( (1,0) \), and a down-step \( (1,-1) \). In other
words, a Motzkin path is a lattice path with the step set
\( \{ (1,1), (1,0), (1,-1)\} \). A \emph{Motzkin--Schr\"oder path}
defined in \cite{KS2023} is a lattice path with the step set
\( \{ (1,1), (1,0), (1,-1), (0,-1) \} \).

We define the \emph{weight function} \( \wt(p) \) to be the product of
weights of steps in the path \( p = (p_0,p_1,\dots, p_n) \), that is,
\( \wt(p) = \prod_{i=1}^n \wt(p_{i-1}\to p_i) \), where
\( \wt(p_{i-1}\to p_i) \) denotes the weight of the step
\( p_{i-1}\to p_i \).

\subsection{First combinatorial model}
\label{sec:general-case}

\begin{defn}\label{def:Luka}
  A \emph{\L{}ukasiewicz path} is a lattice path with the step set
  \( \{ (1,k) \in \ZZ^2 : k \leq 1 \} \). We denote by
  \( \Luka_{n,r,s} \) the set of \L{}ukasiewicz paths from \( (0,r) \)
  to \( (n,s) \). For simplicity, we also define
  \begin{align*}
    \Luka_n := \Luka_{n,0,0}, \qand    \Luka_{n,m} := \Luka_{n,0,m}.
  \end{align*}
  See \Cref{fig:Luka} for an example of a \L{}ukasiewicz path. The
  weight function \( \wtL(p) \) is defined by the weight of a
  \L{}ukasiewicz path \( p \), where the weight of each step is given
  by
  \begin{align*}
    \wtL((a,b)\to (a+1,b+1)) &= 1, \\
    \wtL((a,b) \to (a+1,b-k)) &= -\alpha_b \overline{\alpha_{b-k-1}}\cdot\prod_{j=b-k}^{b-1}(1-|\alpha_j|^2),\quad \mbox{ for \( k = 0,1, \dots ,b  \)}.
  \end{align*}
  Here, we set \( \prod_{j=b-k}^{b-1}(1-|\alpha_j|^2)=1 \) if
  \( k=0 \).
\end{defn}

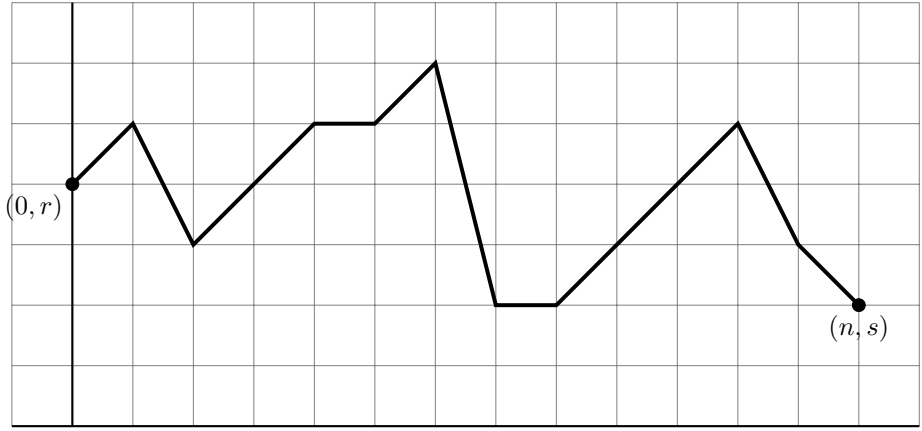
\begin{figure}
  \centering
  \begin{tikzpicture}[scale = 0.8]
    \draw[help lines] (-1,0) grid (14,7);
    \draw[line width = 0.8pt] (0,0)--(0,7);
    \draw[line width = 0.8pt] (-1,0)--(14,0);
    \filldraw[black] (0,4) circle (3pt) node[below left]{\( (0,r) \)};
    \filldraw[black] (13,2) circle (3pt) node[below]{\( (n,s) \)};
    \draw[line width = 1.5pt] (0,4) -- ++(1,1) -- ++(1,-2) -- ++(1,1) -- ++(1,1) -- ++(1,0) -- ++(1,1) -- ++(1,-4) -- ++(1,0) -- ++(1,1) -- ++(1,1) -- ++(1,1) -- ++(1,-2) -- ++(1,-1);
  \end{tikzpicture}
  \caption{A \L{}ukasiewicz path from \( (0,r) \) to \( (n,s) \) where \( n=13 \), \( r=4 \), and \( s=2 \).}
  \label{fig:Luka}
\end{figure}

\begin{thm}\label{thm:Luka}
  For nonnegative integers \( n,r \), and \( s \), we have
  \[
    \mu_{n,r,s} = \sum_{p \in\Luka_{n,r,s}}\wtL(p).
  \]
\end{thm}
\begin{proof}
  When \( n = 0 \), we first see that
  \[
    \mu_{0,r,s} = \frac{\langle \Phi_s,\Phi_r \rangle}{\langle \Phi_s,\Phi_s \rangle} = \delta_{r,s}.
  \]
  We now assume that \( n > 0 \).
  By \eqref{eq:rec1}, we have
  \begin{equation}\label{eq:use_rec1}
    \langle \Phi_s,z^n \Phi_r \rangle = \langle z\Phi_{s-1} , z^n\Phi_r \rangle - \overline{\alpha_{s-1}} \langle \Phi_{s-1}^*,z^n\Phi_r \rangle.
  \end{equation}
  Recall \eqref{eq:rec2} that
  \begin{equation}\label{eq:rec_star}
    \Phi_m^*(z) = \alpha_m z \Phi_m(z) + \Phi_{m+1}^*(z).
  \end{equation}
  By plugging \eqref{eq:rec_star} into \eqref{eq:use_rec1} iteratively
  to increase the index of \( \Phi^* \), we obtain
  \[
    \langle \Phi_s,z^n \Phi_r \rangle = \langle z\Phi_{s-1} , z^n\Phi_r \rangle - \overline{\alpha_{s-1}} \left( \sum_{k = s-1}^{n+r-1} \alpha_k \langle z\Phi_k,z^n \Phi_r \rangle + \langle \Phi_{n+r}^*,z^n\Phi_r \rangle  \right) .
  \]
  The last term in the parenthesis vanishes since
  \begin{align*}
    \langle \Phi_{n+r}^*,z^n\Phi_r \rangle 
    &= \langle z^{n+r} \overline{\Phi_{n+r}}(1/z),z^n\Phi_r(z) \rangle \\
    &= \langle z^{-n} \overline{\Phi_r}(1/z) , z^{-n-r}\Phi_{n+r}(z) \rangle \\ 
    &= \langle \Phi_r^*, \Phi_{n+r} \rangle,
  \end{align*}
  and \( \deg \Phi_r^*(z) < \deg \Phi_{n+r}(z) \) for \( n>0 \). We
  can rewrite \( \langle \Phi_s,z^n \Phi_r \rangle \) as
  \[
    \langle \Phi_s,z^n \Phi_r \rangle = \langle \Phi_{s-1} , z^{n-1}\Phi_r \rangle -  \sum_{k = s}^{n+r-1} \alpha_k\overline{\alpha_{s-1}} \langle \Phi_k,z^{n-1} \Phi_r \rangle .
  \]
  Hence, we obtain the recurrence relation for
  \( \mu_{n,r,s} = {\langle \Phi_s,z^n \Phi_r \rangle}/{\langle
    \Phi_s,\Phi_s \rangle} \):
  \[
    \mu_{n,r,s} = \mu_{n-1,r,s-1} - \sum_{k = s}^{n+r-1}\alpha_k \overline{\alpha_{{s-1}}} \prod_{j=s}^{k-1}(1-|\alpha_j|^2) \mu_{n-1,r,k}.  
  \]
  Let \( \hat{\mu}_{n,r,s} \) be the weight sum of \L{}ukasiewicz
  paths from \( (0,r) \) to \( (n,s) \). Then \( \hat{\mu}_{n,r,s} \)
  satisfies the same recurrence relation for \( \mu_{n,r,s} \) with
  the same initial condition
  \( \hat{\mu}_{0,r,s} = \mu_{0,r,s} = \delta_{r,s} \), which
  completes the proof.
\end{proof}

For instance, there are five \L{}ukasiewicz paths from \( (0,0) \) to
\( (3,0) \):
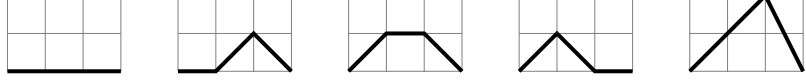
\[
  \begin{tikzpicture}[scale = 0.5]
    \draw[help lines] (0,0) grid (3,2);
    \draw[line width = 1.5pt] (0,0) -- ++(1,0) -- ++(1,0) -- ++(1,0);
  \end{tikzpicture}\qquad
  \begin{tikzpicture}[scale = 0.5]
    \draw[help lines] (0,0) grid (3,2);
    \draw[line width = 1.5pt] (0,0) -- ++(1,0) -- ++(1,1) -- ++(1,-1);
  \end{tikzpicture}\qquad
  \begin{tikzpicture}[scale = 0.5]
    \draw[help lines] (0,0) grid (3,2);
    \draw[line width = 1.5pt] (0,0) -- ++(1,1) -- ++(1,0) -- ++(1,-1);
  \end{tikzpicture}\qquad
  \begin{tikzpicture}[scale = 0.5]
    \draw[help lines] (0,0) grid (3,2);
    \draw[line width = 1.5pt] (0,0) -- ++(1,1) -- ++(1,-1) -- ++(1,0);
  \end{tikzpicture}\qquad
  \begin{tikzpicture}[scale = 0.5]
    \draw[help lines] (0,0) grid (3,2);
    \draw[line width = 1.5pt] (0,0) -- ++(1,1) -- ++(1,1) -- ++(1,-2);
  \end{tikzpicture}.
\]
The sum of their weights is 
\[
  \alpha_0^3 + \alpha_0\alpha_1(1-|\alpha_0|^2) - \alpha_1^2 \overline{\alpha_0}(1-|\alpha_0|^2) +\alpha_0\alpha_1(1-|\alpha_0|^2) + \alpha_2(1-|\alpha_0|^2)(1-|\alpha_1|^2),
\]
which is equal to \( \mu_3 \) computed in \eqref{eq:mu3}.

An expression \( X \) is said to be \emph{\( S \)-positive} if \( X \)
can be written as a polynomial in variables in \( S \) with all
positive coefficients. By~\Cref{thm:Luka}, we also obtain that
\( \mu_{n,r,s} \) is
\( \{\alpha_0,-\overline{\alpha_0}, \alpha_1, -\overline{\alpha_1},
\dots \} \)-positive.

As a corollary of~\Cref{thm:Luka}, we have
\begin{align*}
  \mu_n =  \sum_{p \in \Luka_n} \wtL(p), \qand \mu_{n,m} = \sum_{p \in \Luka_{n,m}} \wtL(p).
\end{align*}

The reciprocity of the generalized moments can be proved using the
conjugate symmetry of the inner product:
\[
  \mu_{-n,r,s}
  = \frac{\langle \Phi_s, z^{-n}\Phi_r \rangle}{\langle \Phi_s, \Phi_s \rangle}
  = \frac{\langle z^n \Phi_s, \Phi_r \rangle}{\langle \Phi_s, \Phi_s \rangle}
  = \frac{\overline{\langle \Phi_r,z^n \Phi_s \rangle}}{\langle \Phi_s, \Phi_s \rangle}
  = \overline{\mu_{n,s,r}} \frac{\langle \Phi_r,\Phi_r \rangle}{\langle \Phi_s,\Phi_s \rangle}.
\]
Using Szeg\H{o}'s recurrence relation \eqref{eq:rec1}, it is also
possible to provide a combinatorial interpretation for
\( \mu_{-n,r,s} \) for nonnegative integers \( n \), \( r \), and
\( s \). Using the lattice path models for \( \mu_{n,r,s} \) and
\( \mu_{-n,r,s} \), we also show the reciprocity
\( \mu_{-n} = \overline{\mu_n} \) in a combinatorial way.

\begin{thm}\label{thm:recipociry}
  For nonnegative integers \( n \), \( r \), and \( s \), we have
  \[
    \mu_{-n,r,s} = \overline{\mu_{n,s,r}}\cdot \frac{\prod_{j=0}^{r-1}(1-|\alpha_j|^2)}{\prod_{j=0}^{s-1}(1-|\alpha_j|^2)}.
  \]
  In particular, \( \mu_{-n} = \overline{\mu_n} \).
\end{thm}
\begin{proof}
  We first recall that \( \mu_{0,r,s} = \delta_{r,s} \). We claim that
  \begin{equation}\label{eq:rec mu-n}
    \mu_{-n,r,s} = (1-|\alpha_{r-1}|^2)\mu_{-(n-1),r-1,s} + \sum_{k=r}^{n+r-1}(-\overline{\alpha_k}\alpha_{r-1})\mu_{-(n-1),k,s}.
  \end{equation}
  Using \eqref{eq:rec1}, we have
  \[
    \langle \Phi_s, z^{-n}\Phi_r \rangle
    = \langle \Phi_s,z^{-(n-1)}\Phi_{r-1} \rangle - \alpha_{r-1} \langle \Phi_s,z^{-n}\Phi_{r-1}^* \rangle.
  \]
  Similar to the proof of \Cref{thm:Luka}, by \eqref{eq:rec2}, we
  obtain
  \begin{multline}\label{eq:rec mu-n 1}
    \langle \Phi_s, z^{-n}\Phi_r \rangle
    = (1-|\alpha_{r-1}|^2)\langle \Phi_s,z^{-(n-1)}\Phi_{r-1} \rangle \\
    + \sum_{k=r}^{n+s-1} (-\alpha_{r-1}\overline{\alpha_k})\langle \Phi_s,z^{-(n-1)}\Phi_k \rangle 
    - \alpha_{r-1} \langle \Phi_s,z^{-n}\Phi_{n+s}^* \rangle.
  \end{multline}
  The last term vanishes since
  \[
    \langle \Phi_s,z^{-n}\Phi_{r+s}^* \rangle = \langle \Phi_s, z^{-n}z^{n+s}\overline{\Phi_{n+s}}(1/z) \rangle = \langle \Phi_{n+s},z^s \overline{\Phi_s}(1/z) \rangle = \langle \Phi_{n+s},\Phi_s^* \rangle = 0,
  \]
  for \( n > 0 \). Dividing both sides of \eqref{eq:rec mu-n 1} by
  \( \langle \Phi_s,\Phi_s \rangle \) gives \eqref{eq:rec mu-n} as
  desired.

  The recurrence relation~\eqref{eq:rec mu-n} with the initial
  condition \( \mu_{0,r,s} = \delta_{r,s} \) gives us that
  \( \mu_{-n,r,s} \) is equal to the weight sum of lattice paths from
  \( (0,s) \) to \( (-n,r) \) with the step set
  \( \{(-1,1),(-1,0),(-1,-1), \dots \} \), where weights of steps are
  given by
  \begin{align*}
    \wt((a,b)\to (a-1,b+1)) &= 1-|\alpha_b|^2 ,\\
    \wt((a,b)\to (a-1,b-k)) &= -\overline{\alpha_b}\alpha_{b-k-1}, \quad\mbox{for \( k=0,1, \dots ,b \).} 
  \end{align*}
  One can see that \( \mu_{-n,r,s} \) is equivalent to taking the
  conjugate on the weights of \L{}ukasiewicz paths for
  \( \mu_{n,s,r} \) multiplied by the factor
  \(
  \prod_{j=0}^{r-1}(1-|\alpha_j|^2)/\prod_{j=0}^{s-1}(1-|\alpha_j|^2)
  \), which completes the proof. At last, we obtain
  \( \mu_{-n}=\overline{\mu_{n}} \) when \( r=s=0 \).
\end{proof}

\subsection{Second combinatorial model}
\label{sec:second-comb-model}

\begin{defn}\label{def:gMot}
  A \emph{gentle Motzkin path} is a Motzkin path satisfying the
  following conditions:
  \begin{itemize}
  \item an up-step \( (a,b) \to (a+1,b+1) \) occurs only when \( a+b \) is even, and
  \item a down-step \( (a,b) \to (a+1,b-1) \) occurs only when \( a+b \) is odd.
  \end{itemize}
  We denote by \( \gMot((a,b)\to (c,d)) \) the set of gentle Motzkin
  paths from \( (a,b) \) to \( (c,d) \). The weight function
  \( \wtM(p) \) is defined by the weight of a gentle Motzkin path
  \( p \), where the weight of each step is given by
  \begin{align*}
   \wtM((a,b)\to (a+1,b+1)) &= 1, \\
    \wtM((a,b)\to (a+1,b-1)) &= (1-|\alpha_{b-1}|^2), \\
    \wtM((a,b)\to (a+1,b)) &= \begin{cases}
                              \alpha_i & \mbox{if \( a+b \) is even},\\
                              -\overline{\alpha_{i-1}} & \mbox{if \( a+b \) is odd.}
                             \end{cases}
  \end{align*}
\end{defn}

See \Cref{fig:gMot} for an example of a gentle Motzkin path. The gray
lines represent the locations where the gentle Motzkin path can lie
on.

\begin{figure}
  \centering
  \begin{tikzpicture}[scale = 0.45]
    \foreach \y in {0,1,...,7}{\draw[help lines] (-5,\y)--(25,\y);}
    \draw[help lines] (-5,1)--(1,7) (-5,3)--(-1,7) (-5,5)--(-3,7) (20,0)--(25,5) (22,0)--(25,3) (24,0)--(25,1);
    \draw[help lines] (-5,6)--(1,0) (-5,4)--(-1,0) (-5,2)--(-3,0) (20,7)--(25,2) (22,7)--(25,4) (24,7)--(25,6);
    \foreach \x in {-4,-2,...,18}{\draw[help lines] (\x,0)--(\x+7,7);}
    \foreach \x in {-4,-2,...,18}{\draw[help lines] (\x,7)--(\x+7,0);}
    \draw[line width = 0.8pt] (0,0)--(0,7);
    \draw[line width = 0.8pt] (-5,0)--(25,0);
    \foreach \x in {-5,-4,...,25}{
      \foreach \y in {0,1,...,7}{
        \filldraw[gray] (\x,\y) circle (1.6pt);
      }}
    \filldraw[black] (-4,4) circle (3pt) node[below]{\( (-r,r) \)};
    \filldraw[black] (24,2) circle (3pt) node[below]{\( (2n-s,s) \)};
    \draw[line width = 1.5pt] (-4,4) -- ++(1,1) -- ++(1,0) -- ++(1,-1) -- ++(1,-1) -- ++(1,0) -- ++(1,1) -- ++(1,1) -- ++(1,0) -- ++(1,0) -- ++(1,1) -- ++(1,0) -- ++(1,-1) -- ++(1,-1) -- ++(1,-1) -- ++(1,-1) -- ++(1,0) -- ++(1,0) -- ++(1,0) -- ++(1,1) -- ++(1,1) -- ++(1,1) -- ++(1,0) -- ++(1,-1) -- ++(1,-1) -- ++(1,0) -- ++(1,0) -- ++(1,-1) -- ++(1,0);
  \end{tikzpicture}
  \caption{A gentle Motzkin path from \( (-r,r) \) to \( (2n-s,s) \) where \( n=13 \), \( r=4 \), and \( s=2 \).}
  \label{fig:gMot}
\end{figure}
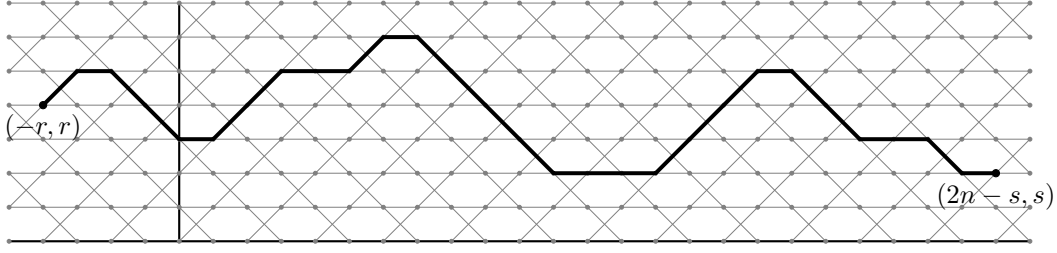

\begin{thm}\label{thm:gentle Motzkin}
  For nonnegative integers \( n,r \), and \( s \), we have
  \[
    \mu_{n,r,s} = \sum_{p \in \gMot((-r,r)\to (2n-s,s))}\wtM(p). 
  \]
\end{thm}
\begin{proof}
  We show that there is a bijective map
  \( \pi: \Luka_{n,r,s} \to \gMot((-r,r)\to (2n-s,s)) \) that
  preserves the weights. For a given \( L \in \Luka_{n,r,s} \), we
  construct \( \pi(L) \) as follows: Start at the point \( (-r,r) \).
  Read the steps of \( L \) from left to right. If a step is an
  up-step, then draw an up-step. If a step is an \( \ell \)-down-step
  \( (1,-\ell) \), then draw a horizontal step, and \( \ell \)
  consecutive down-steps, and finally draw a horizontal step. For
  example, the map \( \pi \) sends the \L{}ukasiewicz path in
  \Cref{fig:Luka} to the gentle Motzkin path in \Cref{fig:gMot}.

  Let \( (1,-\ell_1),\dots,(1,-\ell_k) \) be all the steps in \( L \)
  that are not up-steps. Then there are
  \( \ell_1 + \cdots + \ell_k -r + s \) up-steps, so we have
  \( n = \ell_1 + \cdots + \ell_k - r + s + k \). By the construction,
  \( \pi(L) \) is a lattice path starting from \( (-r,r) \) that has
  \( \ell_1 + \cdots + \ell_k -r + s \) up-steps, \( 2k \) horizontal
  steps, and \( \ell_1 + \cdots + \ell_k \) down-steps. Hence the
  length of \( \pi(L) \) is \( 2n+r-s \) and the height of the ending
  point of \( \pi(L) \) is \( s \). This shows that \( \pi(L) \) is a
  lattice path from \( (-r,r) \) to \( (2n-s,s) \).

  The parity of the sum of the \(x\) and \(y\) coordinates at the
  endpoints of an up-step and a down-step is the same; however, this
  does not hold for a horizontal step. From this observation and the
  construction of \( \pi(L) \), one can see that \( \pi(L) \)
  satisfies the condition of a gentle Motzkin path. Moreover, the
  weight of the step \( (a,b)\to (a+1,b-k) \) in \( L \) is
  \( -\alpha_b \overline{\alpha_{b-k-1}}
  \prod_{j=b-k}^{b-1}(1-|\alpha_j|^2) \). Assigning weights
  \( \alpha_b, (1-|\alpha_{b-1}|^2),
  (1-|\alpha_{b-2}|^2),\dots,(1-|\alpha_{b-k}|^2),
  -\overline{\alpha_{b-k-1}} \) to the corresponding steps of
  \( \pi(L) \) in order yields the weight defined in \Cref{def:gMot}.

  Each stage of the construction of \( \pi \) is invertible, making
  the map \( \pi \) a bijection, which completes the proof.
\end{proof}

\begin{remark}\label{rem:GGTandCMV}
  There is a matrix representation for the generalized moments
  \( \mu_{n,r,s} \). Let \( U = (u_{i,j})_{i,j \geq 0} \) be an
  infinite matrix given by
  \[
    U    =
\begin{bmatrix}
\alpha_0 & 1 & 0 & 0 & \cdots\\
  \alpha_1\rho_0 & -\alpha_1 \overline{\alpha_0} & 1 & 0 & \cdots\\
\alpha_2\rho_0\rho_1 & -\alpha_2 \overline{\alpha_0}\rho_1 & -\alpha_2 \overline{\alpha_1} & 1 & \cdots\\
\alpha_3\rho_0\rho_1\rho_2 & -\alpha_3 \overline{\alpha_0}\rho_1\rho_2 & -\alpha_3 \overline{\alpha_1}\rho_2 & -\alpha_3 \overline{\alpha_2} & \cdots\\
\vdots & \vdots & \vdots & \vdots & \ddots
\end{bmatrix}
  \]
  where \( \rho_i = 1-|\alpha_i|^2 \) for all \( i \). Each
  \( (i,j) \)-entry of the matrix \( U \) is the weight of a step
  \( (a,i) \to (a+1,j) \) in a \L{}ukasiewicz path. Thus, the weight
  sum of \L{}ukasiewicz paths from \( (0,r) \) to \( (n,s) \) is equal
  to the \( (r,s) \)-entry of the matrix \( U^n \).

  There is another matrix representation. Let
  \( C = (c_{i,j})_{i,j\geq 0} \) be an infinite matrix given by
  \[
    C = M \cdot L =
    \begin{bmatrix}
1 & ~ & ~ & ~\\
~ & \Theta_1 & ~ & ~\\
~ & ~ & \Theta_3 & ~\\
~ & ~ & ~ & \ddots
    \end{bmatrix}
    \cdot
    \begin{bmatrix}
\Theta_0 & ~ & ~ & ~\\
~ & \Theta_2 & ~ & ~\\
~ & ~ & \Theta_4 & ~\\
~ & ~ & ~ & \ddots
    \end{bmatrix}
  \]
  where \( \Theta_j = \begin{pmatrix}
\alpha_j & 1\\
1-|\alpha_j|^2 & -\overline{\alpha_j}
\end{pmatrix}
\). Each \( (i,j) \)-entry of the matrix \( L \) is the weight of a
step \( (2k,i)\to (2k+1,j) \) and that of \( M \) is the weight of a
step \( (2k+1,i)\to (2k+2,j) \) in a gentle Motzkin path. Therefore,
we obtain that the weight sum of gentle Motzkin paths from
\( (-r,r) \) to \( (2n-s,s) \) is equal to the \( (r,s) \)-entry of
the matrix \( \left(A_{2n-s}A_{2n-s-1}\cdots A_{-r+1}A_{-r}\right) \),
where \( A_{2k-1} = M \) and \( A_{2k} = L \) for \( k \in \ZZ \).

Note that the matrices \( U \) and \( C \) appear in the literature in
slightly different forms. Using Szeg\H{o}'s recurrence relations, one
can check that
  \[
    U \cdot
    \overline
    {
      \begin{bmatrix}
        \Phi_0(z)\\
        \Phi_1(z)\\
        \Phi_2(z)\\
        \Phi_3(z)\\
        \vdots
      \end{bmatrix}
    }
=
z^{-1}\cdot
\overline{\begin{bmatrix}
\Phi_0(z)\\
\Phi_1(z)\\
\Phi_2(z)\\
\Phi_3(z)\\
\vdots
\end{bmatrix}},
\qquad\mbox{and}\qquad
    C \cdot
    \begin{bmatrix}
\xi_0(z)\\
\xi_1(z)\\
\xi_2(z)\\
\xi_3(z)\\
\vdots
\end{bmatrix}
=
z^{-1}\cdot
\begin{bmatrix}
\xi_0(z)\\
\xi_1(z)\\
\xi_2(z)\\
\xi_3(z)\\
\vdots
\end{bmatrix},
  \]
  where \( \xi_{2n}(z) = z^{-n}\Phi_{2n}(z) \) and
  \( \xi_{2n-1}(z) = z^{-n}\Phi_{2n-1}^*(z) \). That is, the matrices
  \( U \) and \( C \) represent the actions of multiplying
  \( z^{-1} \) on the bases \( \{ \overline{\Phi_n(z)} \}_{n\ge0} \)
  and \( \{\xi_n(z) \}_{n\ge0} \), respectively. Note that
  \( \{ \Phi_n(z) \}_{n \geq 0} \) is a basis in the linear space
  \( \CC[z] \) of polynomials in the variable \( z \) with complex
  coefficients, while \( \{ \overline{\Phi_n(z)}\}_{n \geq 0} \) is a
  basis in the linear space \( \CC[z^{-1}] \) in the variable \( z \)
  with complex coefficients. On the other hand,
  \( \{ \xi_n(z)\}_{n \geq 0} \) is a basis in the linear space of
  Laurent polynomials \( \CC[z,z^{-1}] \). The conjugate transpose of
  \emph{GGT matrix} \( \mathcal{G} \) defined in
  \cite[(4.1.6)]{Simon2005a} and \emph{CMV matrix} \( \mathcal{C} \)
  defined in \cite[Theorem~4.2.5]{Simon2005a} are the conjugate
  orthonormal versions of \( U \) and \( C \). In other words, they
  act on orthonormal bases \( \{\varphi_n(z)\}_{n\ge0} \) and
  \( \{\chi_n(z)\}_{n\ge0} \), respectively, where
  \( \chi_{2n}(z) = z^{-n}\varphi^*_{2n}(z) \) and
  \( \chi_{2n-1}(z) = z^{-n+1}\varphi_{2n-1}(z) \). There are various
  applications of GGT and CMV matrices in the theory of OPUC. For more
  details, see \cite[Chapter~4]{Simon2005a}.
\end{remark}

\subsection{Third combinatorial model}
\label{sec:third-comb-model}

In this subsection, we assume that \( 0 < |\alpha_n| < 1 \) for all
\( n \). Combining \eqref{eq:rec1} and~\eqref{eq:rec2}, we obtain the
three-term recurrence relation for OPUC:
\begin{equation}\label{eq:rec3}
  \Phi_{n+1}(z)=\left( z+\frac{\overline{\alpha_n}}{\overline{\alpha_{n-1}}} \right)\Phi_n(z)-\frac{\overline{\alpha_n}}{\overline{\alpha_{n-1}}}(1-|\alpha_{n-1}|^2)
  z\Phi_{n-1}(z).
\end{equation}
We provide another combinatorial interpretation for \( \mu_{n,r,s} \)
using Schr\"oder paths.

\begin{defn}\label{def:Sch}
  A \emph{Schr\"oder path} is a lattice path with the step set
  \( \{ (1,1), (1,0), (0,-1)\} \). We denote by \( \Sch_{n,r,s} \) the
  set of all Schr\"oder paths from \( (0,r) \) to \( (n,s) \) that do
  not start with a vertical down-step \( (0,-1) \). See \Cref{fig:Sch}
  for an example of a Schr\"oder path. The weight function
  \( \wtS(p) \) is defined by the weight of a Schr\"oder path \( p \),
  where the weight of each step is given by
  \begin{align*}
    \wtS((a,b)\to (a+1,b+1)) &= 1, \\
    \wtS((a,b) \to (a+1,b)) &= -\frac{\overline{\alpha_{b-1}}}{\overline{\alpha_{b}}}, \\
    \wtS((a,b) \to (a,b-1)) &= \frac{\overline{\alpha_{b-2}}}{\overline{\alpha_{b-1}}}(1-|\alpha_{b-1}|^2).
  \end{align*}
\end{defn}

\begin{figure}
  \centering
  \begin{tikzpicture}[scale = 0.8]
    \draw[help lines] (-1,0) grid (14,7);
    \draw[line width = 0.8pt] (0,0)--(0,7);
    \draw[line width = 0.8pt] (-1,0)--(14,0);
    \filldraw[black] (0,4) circle (3pt) node[below left]{\( (0,r) \)};
    \filldraw[black] (13,2) circle (3pt) node[below]{\( (n,s) \)};
    \draw[line width = 1.5pt] (0,4) -- ++(1,1) -- ++(1,0) -- ++(0,-1) -- ++(0,-1) -- ++(1,1) -- ++(1,1) -- ++(1,1) -- ++(0,-1) -- ++(1,1) -- ++(1,0) -- ++(0,-1) -- ++(0,-1) -- ++(0,-1) -- ++(0,-1) -- ++(1,0) -- ++(1,1) -- ++(1,1) -- ++(1,1) -- ++(1,1) -- ++(0,-1) -- ++(0,-1) -- ++(0,-1) -- ++(1,0) -- ++(0,-1);
  \end{tikzpicture}
  \caption{A Schr\"oder path from \( (0,r) \) to \( (n,s) \) where \( n=13 \), \( r=4 \), and \( s=2 \).}
  \label{fig:Sch}
\end{figure}
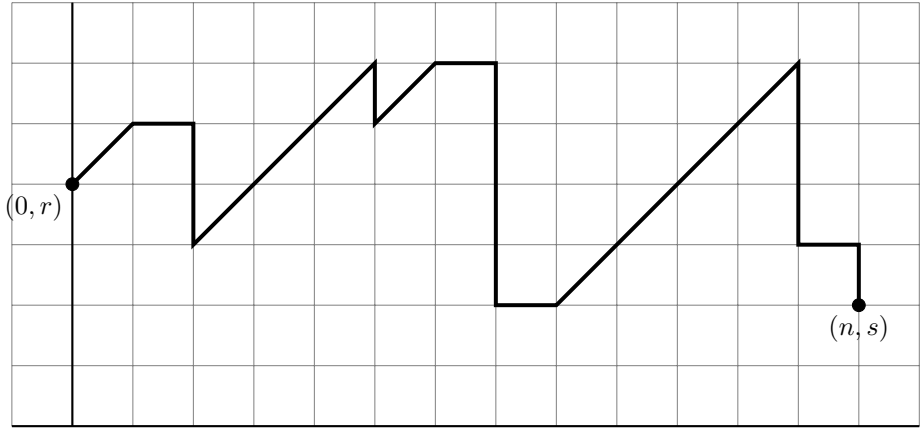

\begin{thm}\label{thm:Sch}
  Assume that \( 0 < |\alpha_n| < 1 \) for all \( n \).
  For nonnegative integers \( n,r \), and \( s \), we have
  \[
    \mu_{n,r,s} = \sum_{p \in \Sch_{n,r,s}}\wtS(p).
  \]
\end{thm}
\begin{proof}
  Using the three-term recurrence relation \eqref{eq:rec3} for OPUC,
  we have
  \begin{align*}
    \mu_{n,r,s+1}
    &= \langle \Phi_{s+1},z^n\Phi_r \rangle /\langle \Phi_{s+1},\Phi_{s+1} \rangle \\
    &= \left\langle \left(z + \frac{\overline{\alpha_s}}{\overline{\alpha_{s-1}}} \right)\Phi_s - \frac{\overline{\alpha_s}}{\overline{\alpha_{s-1}}}(1-|\alpha_{s-1}|^2)z\Phi_{s-1}, z^n\Phi_r \right\rangle /\langle \Phi_{s+1},\Phi_{s+1} \rangle\\
    &= \frac{1}{1-|\alpha_s|^2} \left( \mu_{n-1,r,s} + \frac{\overline{\alpha_s}}{\overline{\alpha_{s-1}}}\mu_{n,r,s} -  \frac{\overline{\alpha_s}}{\overline{\alpha_{s-1}}} \mu_{n-1,r,s-1} \right).
  \end{align*}
  It follows from the recurrence relation for the weight sum of
  Schr\"oder paths with the initial condition
  \( \mu_{0,r,s} = \delta_{r,s} \).
\end{proof}

Both \Cref{thm:Luka} and \Cref{thm:Sch} provide path interpretations
of \( \mu_{n,r,s} \). Consequently, it naturally follows that the
weight sums of Schr\"oder paths and \L{}ukasiewicz paths are equal.
The following proposition provides a combinatorial proof of this. Note
that the cardinality of \( \Sch_{n,r,s} \) is larger than that of
\( \Luka_{n,r,s} \). In the proof, we show that a group of Schr\"oder
paths corresponds to a \L{}ukasiewicz path.

\begin{prop}\label{prop:Luka=Sch}
  For nonnegative integers \( n,r \), and \( s \), we have
  \[
    \sum_{p \in \Sch_{n,r,s}}\wtS(p) = \sum_{q \in \Luka_{n,r,s}}\wtL(q),
  \]
  where \( \wtS \) and \( \wtL \) are the weight functions described
  in \Cref{def:Sch} and \Cref{def:Luka}, respectively.
\end{prop}
\begin{proof}
  Since the weight of each step depends only on the height of the
  starting point, we consider paths as seqeunces
  \( S_{i_1} \dots S_{i_n} \) where \( S \) is one of the up-step
  \( U=(1,1) \), horizontal step \( H = (1,0) \), and vertical
  down-step \( V = (0,-1) \), and \( i_j \) is the initial height of
  the \( j \)th step. We first have
  \[
    \wtS(U_bV_{b+1}) + \wtS(H_b) = 1\cdot(1-|\alpha_b|^2)\frac{\overline{\alpha_{b-1}}}{\overline{\alpha_b}} + \left(-\frac{\overline{\alpha_{b-1}}}{\overline{\alpha_{b}}}\right) = -\alpha_b \overline{\alpha_{b-1}}.
  \]
  Hence, \( \mu_{n,r,s} \) is equal to the weight sum of Schr\"oder
  paths from \( (0,r) \) to \( (n,s) \) with no subpath \( UV \),
  where weights of steps are defined by
  \begin{align*}
    \wt((a,b)\to (a+1,b+1)) &= 1, \\
    \wt((a,b) \to (a+1,b)) &= -\alpha_b \overline{\alpha_{b-1}}, \\
    \wt((a,b) \to (a,b-1)) &= \frac{\overline{\alpha_{b-2}}}{\overline{\alpha_{b-1}}}(1-|\alpha_{b-1}|^2).
  \end{align*}
  Further, since
  \[
    \wt(H_bV_bV_{b-1} \dots V_{b-k+1}) = -\alpha_b \overline{\alpha_{b-1}} \prod_{j=1}^{k}\left( \frac{\overline{\alpha_{b-j-1}}}{\overline{\alpha_{b-j}}}(1-|\alpha_{b-j}|^2) \right) = -\alpha_b \overline{\alpha_{b-k-1}}\prod_{j=b-k}^{b-1}(1-|\alpha_j|^2),
  \]
  we correspond a subpath \( H_bV_bV_{b-1} \dots V_{b-k+1} \) to a
  \( k \)-down-step \( (1,-k) \) starting at height \( b \) to
  complete the proof.
\end{proof}

\begin{remark}\label{rmk:RI_OPUC}
  It is known \cite{Ismail1995, KS2023} that a polynomial sequence
  \( (P_n(x))_{n\ge0} \) is a sequence of orthogonal polynomials of
  type \( R_I \) if and only if there exist
  \( (a_n)_{n\ge 1}, (b_n)_{n\ge0}, \) and \( (\lambda_n)_{n\ge 1} \)
  such that
  \begin{align}
    \label{eq:R1_rec}
    P_{n+1}(x) = (x - b_n)P_n(x) - (a_nx + \lambda_n) P_{n-1}(x), \quad n \geq 0,
  \end{align}
  where \( a_n \ne 0 \) and \( P_n(-\lambda_n/a_n) \ne 0 \), with the
  initial conditions \( P_{-1}(x) =0 \) and \( P_0(x) = 1 \). It is
  natural to ask about the relation between orthogonal polynomials of
  type \( R_I \) and OPUC once we recognize that \eqref{eq:rec3} is of
  the form as~\Cref{eq:R1_rec}. Indeed, the polynomial sequence
  \( (\Phi_n(z))_{n\ge0} \) satisfying~\eqref{eq:rec3} can be regarded
  as a special case of orthogonal polynomials of type \( R_I \) with
  respect to the linear functional \( \widehat{\cL} \) such that
\begin{equation}\label{eq:R1_ortho}
  \widehat{\cL}\left(z^{n}\frac{\Phi_m(z)}{d_m(z)}\right)=0, \quad 0\le n<m,
\end{equation}
where
\( d_m(z) = \prod_{i=1}^{m}\left(
  \frac{\overline{\alpha_i}}{\overline{\alpha_{i-1}}}(1-|\alpha_{i-1}|^2)z
\right) \). On the other hand, by~\Cref{thm:linear functional}, there
exists a linear functional \( \cL \) that satisfies the orthogonality
condition \eqref{eq:ortho_cond}. As noted in \Cref{rmk:LBP_OPUC}, for
the OPUC \( (\Phi_n(z))_{n\ge0} \) with respect to \( \cL \), the
orthogonality condition is equivalent to the condition
\begin{equation}\label{eq:half_ortho}
  \cL(z^{-n}\Phi_m(z)) = \kappa_n\delta_{n,m}, \quad 0 \le n \le m.  
\end{equation}
According to \cite[Proposition 4.4]{KS2023}, a relation between the
linear functionals \( \widehat{\cL} \) and \( \cL \) is given by
  \begin{equation}\label{eq:relation_two_L}
    \cL(f(z))=\overline{\alpha_0}\widehat{\cL}(z^{-1}f(z))
  \end{equation}
  for any \( f(z)\in V \). This relation allows us to
  recover~\eqref{eq:R1_ortho} from~\eqref{eq:ortho_cond}. Thus, when
  \( \alpha_n \ne 0 \), OPUC can be viewed as a modified (but
  essentially the same) version of a special case of orthogonal
  polynomials of type \( R_I \). Note that it is not true when
  \( \alpha_k = 0 \) for some \( k \).
\end{remark}

\begin{remark}
  In \cite{KS2023}, the authors showed that
  \(
  \widehat{\mu}_{n,r,s}:=\widehat{\cL}\left(z^n\Phi_r(z)\frac{\Phi_s(z)}{d_s(z)}\right)
  \) is the weight sum of Schr\"oder paths from \( (0,r) \) to
  \( (n,s) \), where the weight of each step is given as
  in~\Cref{def:Sch}. In \cite{KS2023}, they denote this as
  \( \mu_{n,r,s} \). Since
  \[
   d_s(z) = \frac{\overline{\alpha_s}}{\overline{\alpha_0}}\langle \Phi_s, \Phi_s \rangle_{\cL} z^s, \qand z^{-s}\Phi_s(z)=\overline{\Phi_s^*}(1/z),
  \]
  by~\eqref{eq:relation_two_L}, we obtain
  \begin{equation}\label{eq:Sch_KS}
    \widehat{\mu}_{n,r,s}=\widehat{\cL}\left(z^n\Phi_r(z)\frac{\Phi_s(z)}{d_s(z)}\right) = \frac{\overline{\alpha_0}L_1(z^{n}\Phi_r(z)\overline{\Phi_s^*}(1/z))}{\overline{\alpha_s}\langle \Phi_s, \Phi_s \rangle_{\cL}} = \frac{1}{\overline{\alpha_s}}\cdot\frac{\langle z^{n+1}\Phi_r, \Phi_s^* \rangle_{\cL}}{\langle \Phi_s, \Phi_s \rangle_{\cL}}.
  \end{equation}
  We revisit \( \widehat{\mu}_{n,r,s} \) in~\Cref{sec:generalized_lc}.
\end{remark}

\section{Explicit formulas of generalized moments for some families of OPUC}
\label{sec:examples}

In this section, we provide explicit formulas for \( \mu_{n,m} \) and
\( \mu_{n,r,s} \) for various examples of OPUC, as an application of
combinatorial interpretations for the generalized moments. Since all
examples in this section satisfy~\eqref{eq:rec1}, we classify them
according to their Verblunsky coefficients
\( (\alpha_n)_{n \geq 0} \). Most of these examples are introduced in
\cite{Ismail2009} or \cite{Simon2005a}.

\subsection{Geronimus polynomials}
For \( \alpha\in \mathbb{D} \), the Geronimus polynomials are defined
by the Verblunsky coefficients
\[
  \alpha_n = \alpha, \quad n \geq 0.
\]
For more information concerning Geronimus polynomials, see \cite[Example~1.6.12]{Simon2005a}.

Suppose \( \alpha = 0 \), which gives the simplest example
\( \Phi_n(z) = z^n \) of OPUC. This implies that
\( \mu_{n,r,s} = \langle z^s,z^{n+r} \rangle = \mu_{n+r-s} \). Using
the \L{}ukasiewicz path model, we see that \( \mu_{n,r,s} \) is the
weight sum of lattice paths from \( (0,0) \) to \( (n+r-s,0) \) where
we only use up-steps with weight \( 1 \). Therefore,
\( \mu_{n,r,s} = \mu_{n+r-s} = \delta_{n+r-s,0} \). Now, consider the
case \( \alpha\ne0 \).

\begin{prop}\label{pro:Geronimus_nrs}
  For the Geronimus polynomials, suppose \( \alpha \ne 0 \). Then, we
  have
  \[
    \sum_{n\ge0}\mu_{n,m}z^n = \frac{2|\alpha|^2(2z)^m}{\left(2|\alpha|^2+\alpha-\alpha z-\alpha\sqrt{1-2z+4|\alpha|^2z+z^2}\right)\left(1+z+\sqrt{1-2z+4|\alpha|^2z+z^2}\right)^m}.
  \]
\end{prop}
\begin{proof}
  Using the Schr\"{o}der path model, we have
  \[
    \mu_{n,m}=\sum_{p \in \Sch_{n,0,m}}\wt(p),
  \]
  where
  \begin{align*}
    \wt((a,b)\to (a+1,b+1)) &= 1, &&\\
    \wt((a,b) \to (a+1,b)) &=
                             \begin{cases}
                               \frac{\alpha}{|\alpha|^2}& \mbox{if \( b=0 \)},\\
                               -1& \mbox{if \( b>0 \),}
                             \end{cases}\\
    \wt((a,b) \to (a,b-1)) &=
                             \begin{cases}
                               -\frac{\alpha}{|\alpha|^2}(1-|\alpha|^2)& \mbox{if \( b=1 \)},\\
                               1-|\alpha|^2& \mbox{if \( b>1 \).}
                             \end{cases}
  \end{align*}
  To obtain the generating function for \( \mu_{n,m} \), we consider
  Schr\"oder paths with another weights. Let \( \lambda_n \) be the
  weight sum of Schr\"oder paths from \( (0,0) \) to \( (n,0) \) where
  \[
    \wt((a,b)\to (a+1,b+1)) = 1, \quad \wt((a,b) \to (a+1,b)) = -1, \qand
    \wt((a,b) \to (a,b-1)) = 1-|\alpha|^2.
  \]
  Let \( f(z) = \sum_{n\ge0}\lambda_nz^n \) and
  \( g(z) = \sum_{n\ge0} \mu_{n,0}z^n \). Then, classifying the
  Schr\"oder paths in terms of the first step, one can have two
  functional equations:
  \begin{align*}
    f(z) &= 1 -zf(z) + (1-|\alpha|^2)zf(z)^2, \\
    g(z) &= 1+\frac{\alpha}{|\alpha|^2}zg(z)-\frac{\alpha}{|\alpha|^2}(1-|\alpha|^2)zf(z)g(z), \end{align*}
  see~\Cref{fig:two_fct_eq}.
  Solving these equations, we obtain
  \[
    g(z) = \frac{2|\alpha|^2}{2|\alpha|^2+\alpha-\alpha z-\alpha\sqrt{1-2z+4|\alpha|^2z+z^2}},\quad f(z) =\frac{2}{1+z+\sqrt{1-2z+4|\alpha|^2z+z^2}}.
  \]
  This completes the proof since
  \( \sum_{n\ge0}\mu_{n,m}z^n = g(z)\left( zf(z) \right)^m \),
  see~\Cref{fig:mu_nm}.
\end{proof}

\begin{figure}
  \centering
    \begin{tikzpicture}
    \begin{scope}[scale=0.7]
      \node at (-3.5,0) {\( f(z) = \)};
      
      \node at (-2,0) [circle, fill,inner sep=.5pt]{};
      
      \node at (-1,0) {\( + \)};
      
      \draw (0,0) -- node [above] {\( -z \)} (1,0);
      \draw (1,0) arc (180:0:1.5) -- cycle;
      \node at (2.5,0.6) {\( f(z) \)};

      \node at (5,0) {\( + \)};
      
      \draw (6,0) -- node [text width=0.3cm,above ] {\( z \)}  (7,1);
      \draw (7,1) arc (180:0:1.5) -- cycle;
      \node at (8.5,1.6) {\( f(z) \)};
      \draw (10,1) --  node [ left ] {\( 1-|\alpha|^2 \)}  (10,0);
      \node at (11.5,0.6) {\( f(z) \)};
      \draw (10,0) arc (180:0:1.5) -- cycle;
      
      \node at (13.5,0) {,};
    \end{scope} 
\begin{scope}[scale=0.7, yshift= -3.5cm]
      \node at (-3.5,0) {\( g(z) = \)};
      
      \node at (-2,0) [circle, fill,inner sep=.5pt]{};
      
      \node at (-1,0) {\( + \)};
      
      \draw (0,0) -- node [above] {\( \frac{\alpha}{|\alpha|^2} z \)} (1,0);
      \draw (1,0) arc (180:0:1.5) -- cycle;
      \node at (2.5,0.6) {\( g(z) \)};

      \node at (5,0) {\( + \)};
      
      \draw (6,0) -- node [text width=0.3cm,above ] {\( z \)}  (7,1);
      \draw (7,1) arc (180:0:1.5) -- cycle;
      \node at (8.5,1.6) {\( f(z) \)};
      \draw (10,1) --  node [ left, font=\footnotesize ] {\( -\frac{\alpha}{|\alpha|^2}(1-|\alpha|^2)\)}  (10,0);
      \node at (11.5,0.6) {\( g(z) \)};
      \draw (10,0) arc (180:0:1.5) -- cycle;
      
      \node at (13.5,0) {.};
    \end{scope} 
  \end{tikzpicture}
  \caption{A visualization of two functional equations for \( f(z) \) and \( g(z) \).} 
  \label{fig:two_fct_eq}
\end{figure}

\begin{figure}
    \centering
  \begin{tikzpicture}
    \begin{scope}[scale=0.7]
      \draw (0,0) arc (0:180:1.5) -- cycle;
      \draw (0,0) -- node [text width=0.3cm,above ] {\( z \)}  (1,1);
      \draw (1,1) arc (180:0:1.5) -- cycle;
      \draw (4,1) -- (5,2);
      \draw (5,2) arc (180:0:1.5) -- cycle;
      \draw (8,2) -- node [text width=0.3cm,above ] {\( z \)} (9,3);
      \draw (9,3) arc (180:0:1.5) -- cycle;      
      \node at (8.6,2.0){$\uparrow$};
      \node at (8.6,1.3){the $m$th up-step};
      \node at (4.4,2) {\reflectbox{$\ddots$}};
      \node at (-1.5,0.6) {\( g(z) \)};
      \node at (2.5,1.6) {\( f(z) \)};
      \node at (6.5,2.8) {\reflectbox{$\ddots$}};
      \node at (10.5,3.6) {\( f(z) \)};
    \end{scope} 
  \end{tikzpicture}
  \caption{A visualization for \( \sum_{n\ge0}\mu_{n,m}z^n = g(z)\left( zf(z) \right)^m \).} 
  \label{fig:mu_nm}
\end{figure}

\begin{remark}
  By \Cref{pro:Geronimus_nrs}, one can derive an explicit formula for
  \( \mu_{n,m} \). Briefly, one can check that
  \begin{align*}
    [z^0]g(z) &= 1,  \quad [z^n]\left( zf(z) \right)^m = \delta_{n,m}\quad\text{for \;} n\le m,\\ [z^n]g(z)&=\sum_{l=0}^n\sum_{k=0}^{n-l}\frac{l}{k+l}\binom{n+k-1}{k,k+l-1,n-k-l}(-|\alpha|^2)^k\alpha^l\quad \text{for \;} n>0,\\
    [z^n]\left( zf(z) \right)^m &= \sum_{k=0}^{n-m}\frac{m}{k+m}\binom{n+k-1}{k,k+m-1,n-k-m}(-1)^{n-k-m}(1-|\alpha|^2)^k\quad \text{for \;} n>m.
  \end{align*}
  Using these, a summation form for
  \( \mu_{n,m} = [z^n]g(z)(zf(z))^m \) is obtained. However, finding a
  simple closed formula without summation is challenging. For
  instance, consider the case when \( \alpha = 1 \). We then have
  \begin{equation}\label{eq:geronimus_simple}
    \mu_{n,m} = [z^n]\frac{1}{1-z}\left(\frac{z}{1+z}\right)^m = \sum_{k=0}^{n-m}\binom{m+k-1}{k}(-1)^k.
  \end{equation}
  Even though this is a very simple case, we cannot expect a simple
  form without summation for~\eqref{eq:geronimus_simple} for general
  \( m \) as it does not factor.
\end{remark}

\begin{remark}\label{rmk:nm_to_nrs}
  Let \( \Phi_n(z)=\sum_{i=0}^np_{n,i}z^i \). If we have explicit
  formulas of \( \mu_{n,m} \) and \( p_{n,i} \), then it is possible
  to obtain an explicit formula of \( \mu_{n,r,s} \) by a direct
  computation:
  \begin{align}
\label{eq:nrs_from_nm}    \mu_{n, r, s} &=\frac{\left\langle\Phi_s , z^n \Phi_r\right\rangle}{\left\langle\Phi_s, \Phi_s\right\rangle}  =\sum_{i=0}^r\overline{p_{r,i}} \frac{\left\langle\Phi_s, z^{n+i}\right\rangle}{\left\langle\Phi_s , \Phi_s\right\rangle} = \sum_{i=0}^r\overline{p_{r,i}} \mu_{n+i, s}.
   \end{align}
   This approach will be used for other examples. For the Geronimus
   case when \( \alpha \not= 0 \), one can check that
  \[
    p_{n,i} = [z^n]\left( 1-\frac{\overline{\alpha}}{1-z} \right)\left( 1+ \frac{|\alpha|^2z}{1+z} \right)^i = \sum_{j=0}^{n-i}\left( \binom{i}{j}-\frac{1}{\alpha}\binom{i}{j-1} \right)\binom{n-i-1}{j-1}|\alpha|^{2j},
  \]
  where \( \binom{-1}{-1} =1 \), by showing that \( p_{n,i} \)
  satisfies the following recurrence relation:
  \[
    p_{n,i} = p_{n-1,i-1} +p_{n-1,i}-(1-|\alpha|^2)p_{n-2,i-1},
  \]
  where \( p_{k,k}=1 \) and \( p_{k,0} = -\overline{\alpha}\) for
  \( k\ge0 \).
\end{remark}

Note that our calculations in this subsection are related to the
Riordan group theory. As a fancy tool inherenting the Lagrange
inversion foumula, the Rioran group theory simplifies many
combinatorial calculations. For those who are intersted in this
theory, we refer to \cite{DFSW2024, SSBCHMW2022}.

\subsection{Bernstein--Szeg\H{o} polynomials}
For \( \zeta \in \DD \), the Bernstein--Szeg\H{o} polynomials are
defined by the Verblunsky coefficients
\[
  \alpha_0 = \zeta \qand \alpha_j = 0, \quad j\geq 1.
\]
From~\eqref{eq:rec1}, one can directly obtain
\( \Phi_n(z)=z^n-\overline{\zeta}z^{n-1} \) for \( n\ge1 \). Since
\( 0 = \cL(\overline{\Phi_n(z)}) =\mu_n - \zeta\mu_{n-1} \) for
\( n\ge1 \), we have
\( \mu_n = \zeta \mu_{n-1} = \dots = \zeta^n\mu_0 = \zeta^n \) for all
\( n \geq 0 \). In general,
\[
  \mu_{n,m} = \frac{\langle \Phi_m,z^n  \rangle}{\langle \Phi_m,\Phi_m \rangle}  = \frac{\mu_{n-m}-\overline{\zeta}\mu_{n-m+1}}{1-|\zeta|^2} = \frac{1-|\zeta|^2}{1-|\zeta|^2}\zeta^{n-m} = \zeta^{n-m} \quad\mbox{for all \( n \geq m \)}.
\]
If \( r\ge1 \), then we have
\[
  \mu_{n,r,s}= \frac{\langle \Phi_s,z^n\Phi_r  \rangle}{\langle \Phi_s,\Phi_s \rangle}  = \frac{\langle \Phi_s,\Phi_{n+r} \rangle}{\langle \Phi_s,\Phi_s \rangle} = \delta_{s,n+r}.
\]
To summarize these results, we have the following proposition.
\begin{prop}\label{pro:BS}
  For the Bernstein--Szeg\H{o} polynomials, we have
  \[
    \mu_{n,m} =
    \begin{cases}
      \delta_{n,m} & \mbox{if \( n\leq m \)},\\
      \zeta^{n-m} & \mbox{if \( n> m \),}
    \end{cases}
    \qand
    \mu_{n,r,s} =
    \begin{cases}
      \mu_{n,s} & \mbox{if \( r=0 \)},\\
      \delta_{s,n+r} & \mbox{if \( r >0 \).}
    \end{cases}
  \]

\end{prop}

\subsection{Single inserted mass point}
For \( \gamma \in (0,1) \), the OPUC that correspond to the single
inserted mass point measure
\( d\mu = (1-\gamma)\frac{d\theta}{2\pi} + \gamma \delta(\theta-0) \),
where \( \delta(\theta-0) \) is a point mass at \( \theta = 0 \), are
defined in \cite[Example~1.6.3]{Simon2005a} by the Verblunsky
coefficients
\[
  \alpha_n = \frac{\gamma}{1+n\gamma}, \quad n \geq 0.
\]
\begin{prop}\label{pro:simp_nm}
  For the OPUC that correspond to the single inserted mass point
  measure, we have
  \begin{align*}
    \mu_{n,m} =
    \begin{cases}
      \delta_{n,m} & \mbox{if \( n \le m \)},\\
      \frac{\gamma}{1+m\gamma} & \mbox{if \( n > m \).}
    \end{cases}
  \end{align*}
\end{prop}
\begin{proof}
  Since \( \alpha_n\ne0 \), we can use Schr\"oder path model to obtain
  the following recurrence:
  \begin{align*}
    \mu_{n,m} &= \mu_{n-1,m-1} - \frac{1+m\gamma}{1+(m-1)\gamma}\mu_{n-1,m} + \frac{1+(m+1)\gamma}{1+m\gamma}\mu_{n,m+1},& &m>0,\\
    \mu_{n,m} &= \frac{1}{\gamma}\mu_{n-1,m} - \frac{1-\gamma^2}{\gamma}\mu_{n,m+1},& &m=0.
  \end{align*}
  From the Schr\"oder path model, it is obvious that
  \( \mu_{n,m} = \delta_{n,m} \) for \( n\le m \). For \( n>m \), let
  \( \lambda_{n,m} = \frac{\gamma}{1+m\gamma} \). Then the result can
  be obtained by checking that \( \lambda_{n,m} \) satisfies the same
  recurrence relation with the same initial
  condition.
\end{proof}
 
\begin{prop}\label{pro:simp_nrs}
  For the OPUC that correspond to the single inserted mass point
  measure, if \( r > 0 \), then we have
  \[
    \mu_{n,r,s} =
    \begin{cases}
      \delta_{n+r,s} & \mbox{if \( s \ge n+r \)},\\
      - \frac{n\gamma^2}{(1+(r-1)\gamma)(1+s\gamma)} & \mbox{if \( n-1 < s < n+r \),} \\
      \frac{(1-\gamma)\gamma}{(1+(r-1)\gamma)(1+s\gamma)}  & \mbox{if \( 0 \leq s \le n-1 \).}
    \end{cases}
  \]
\end{prop}
\begin{proof}
  Suppose \( r > 0 \). It is known \cite[(1.6.6)]{Simon2005a} that
  \( \Phi_n(z) = z^n - \alpha_{n-1}(z^{n-1} + \dots z + 1) \) for
  \( n\ge1 \). By \eqref{eq:nrs_from_nm}, we have
  \begin{align*}
    \mu_{n,r,s} =  \mu_{n+r,s} - \frac{\gamma}{1+(r-1)\gamma} \sum_{i=0}^{r-1} \mu_{n+i,s}.
  \end{align*}
  Using this equation and \Cref{pro:simp_nm}, we can compute
  \( \mu_{n,r,s} \) for each case:
  \begin{itemize}
  \item If \( s\ge n+r \), then the summation on the right-hand side
    vanishes and hence \( \mu_{n,r,s} = \delta_{n+r,s} \).
  \item  If \( n-1 < s < n+r \), then
    \[
      \mu_{n,r,s}= \frac{\gamma}{1+s\gamma}-  \frac{\gamma}{1+(r-1)\gamma}\left( 1+ \frac{(n+r-s-1)\gamma}{1+s\gamma} \right) = - \frac{n\gamma^2}{(1+(r-1)\gamma)(1+s\gamma)}.
    \]
  \item 
    If \( 0 \leq s \le n-1 \), then 
    \[
      \mu_{n,r,s}= \frac{\gamma}{1+s\gamma}-  \frac{\gamma}{1+(r-1)\gamma} \frac{r\gamma}{1+s\gamma} =  \frac{(1-\gamma)\gamma}{(1+(r-1)\gamma)(1+s\gamma)}. \qedhere
    \]
  \end{itemize}
\end{proof}

\subsection{Circular Jacobi polynomials}
For \( a>-1 \), the circular Jacobi polynomials are defined in
\cite[Example~8.2.5]{Ismail2009} by the orthogonality with respect to
the weight functiondm
\[
  w(\theta) = \frac{\Gamma^2(a+1)}{2\pi \Gamma(2a+1)} |1-e^{i\theta}|^{2a},
\]
and are also defined by Verblunsky coefficients 
\[
  \alpha_n = - \frac{a}{n+a+1}, \quad n \geq 0.
\]

\begin{prop}\label{prop:Jacobi_nm}
  For circular Jacobi polynomials, we have
  \[
    \mu_{n,m} = (-1)^{n-m} \binom{n}{m} \frac{(a)_{n-m}}{(a+n)_{n-m}},
  \]
  where \( (a)_k \) is a Pochhammer symbol given by
  \( (a)_k = a(a+1)\cdots (a+k-1) \) for \( k \geq 1 \), and
  \( (a)_0 = 1 \).
\end{prop}
\begin{proof}
  We have the recurrence relation for \( \mu_{n,m} \) from the proof
  of \Cref{thm:Sch} by letting \( r=0 \). Let \( \lambda_{n,m} \) be
  the right-hand side of the equation. It is enough to show that
  \( \lambda_{n,m} \) has the same recurrence relation, that is,
  \[
    \lambda_{n,m} = \lambda_{n-1,m-1} - \frac{a+m+1}{a+m} \lambda_{n-1,m} + \frac{a+m+1}{a+m}\left( 1- \left( \frac{a}{a+m+1} \right)^2  \right)\lambda_{n,m+1}.
  \]
  The initial conditions are the same:
  \[
    \mu_{n,m} = \lambda_{n,m} = \delta_{n,m} \mbox{ for } n\le m \qand \mu_{n,-1} = \lambda_{n,-1} = 0 \quad\mbox{for all \( n\geq 0 \).}
  \]
  The proof follows from
  \begin{align*}
    \lambda_{n-1,m-1} &= \frac{m(a+n)}{n(a+m)} \lambda_{n,m}, \\
    \frac{a+m+1}{a+m} \lambda_{n-1,m} &= \frac{a+m+1}{a+m}\frac{(n-m)(a+n)}{n(a-n+m+1)} \lambda_{n,m}, \\
    \frac{(m+1)(2a+m+1)}{(a+m)(a+m+1)} \lambda_{n,m+1} &=\frac{(m+1)(2a+m+1)}{(a+m)(a+m+1)}  \frac{(n-m)(a+m+1)}{(m+1)(a-n+m+1)} \lambda_{n,m} \\
    &= \frac{(2a+m+1)(n-m)}{(a+m)(a-n+m+1)} \lambda_{n,m},
  \end{align*}
  with a direct computation
  \[
    \frac{m(a+n)}{n(a+m)} + \frac{(a+m+1)(n-m)(a+n)}{(a+m)n(a-n+m+1)} + \frac{(2a+m+1)(n-m)}{(a+m)(a-n+m+1)} = 1.\qedhere
  \]
\end{proof}

\begin{prop}\label{pro:Jacobi_nrs}
  For circular Jacobi polynomials, we have
  \[
    \mu_{n,r,s} = \sum_{i=0}^r(-1)^{n+i-s}\binom{n+i}{s}\binom{r}{i}\frac{(a)_{i+1}(-a)_{n+i-s}}{(a+r-i)_{i+1}(-a-n-i)_{n+i-s}}.
  \]
\end{prop}
\begin{proof}
  It is known \cite[(8.2.22)]{Ismail2009} that
  \begin{align*}
    \phi_n(z)&=\frac{(a)_n}{\sqrt{n!(2a+1)_n}} {}_2F_1(-n,a+1;-n+1-a;z)  \\
             &= \frac{(a)_n}{\sqrt{n!(2a+1)_n}}\sum_{i=0}^n \frac{(-n)_i(a+1)_i}{(-n+1-a)_i}\frac{z^i}{i!}. \qquad (\because (-n)_i=0 \text{ for } i>n.)
  \end{align*}
  Thus, we have
  \begin{align*}
    \Phi_n(z) &= \frac{a}{a+n}\sum_{i=0}^n \frac{(-n)_i(a+1)_i}{(-n+1-a)_i}\frac{z^i}{i!}  \\
              &= \sum_{i=0}^n \frac{(-1)^i(n)(n-1)\cdots(n-i+1)}{(a+n)(-1)^i(a+n-1)(a+n-2)\cdots (a+n-i)}\frac{a(a+1)\cdots(a+i)}{i!}z^i  \\
              &= \sum_{i=0}^n\binom{n}{i}\frac{(a)_{i+1}}{(a+n-i)_{i+1}}z^i.
  \end{align*}
  Using~\eqref{eq:nrs_from_nm}, we obtain
  \begin{align*}
    \mu_{n, r, s}   & =\sum_{i=0}^r\binom{r}{i} \frac{(a)_{i+1}}{(a+r-i)_{i+1}} \mu_{n+i, s} \\
                  & =\sum_{i=0}^r\binom{r}{i} \frac{(a)_{i+1}}{(a+r-i)_{i+1}} (-1)^{n+i-s}\binom{n+i}{s} \frac{(-a)_{n+i-s}}{(-a-n-i)_{n+i-s}} \\
                  & =\sum_{i=0}^r(-1)^{n+i-s}\binom{n+i}{s}\binom{r}{i} \frac{(a)_{i+1}(-a)_{n+i-s}}{(a+r-i)_{i+1}(-a-n-i)_{n+i-s}}.\qedhere
  \end{align*}
\end{proof}

\subsection{Single nontrivial moment}
For the OPUC \cite[Example~1.6.4]{Simon2005a} that correspond to the
single nontrivial moment measure
\( d\mu_a(\theta) = [1- a\cos(\theta)]\frac{d\theta}{2\pi} \) where
\( 0 < a \le 1 \), the Verblunsky coefficients depend on \( a \). For
\( a=1 \), such OPUC are defined by the Verblunsky coefficients
\[
  \alpha_n = - \frac{1}{n+2}, \quad n \geq 0,
\]
which coincide with the Verblunsky coefficients of circular Jacobi
polynomials for \( a = 1 \). By plugging \( a=1 \)
into~\Cref{prop:Jacobi_nm} and~\Cref{pro:Jacobi_nrs}, we obtain the
following proposition.
\begin{prop}\label{pro:SNM_nm_a1}
  For the OPUC that correspond to the single nontrivial moment measure
  with \( a = 1 \), we have
  \( \Phi_r(z)=\sum_{i=0}^r \frac{i+1}{r+1} z^i \). Furthermore,
  \[
    \mu_{n,m} =
    \begin{cases}
      1 & \mbox{if \( m=n \)},\\
      - \frac{n}{n+1} & \mbox{if \( m=n-1 \)},\\
      0 & \mbox{otherwise,}
    \end{cases} \qand
    \mu_{n,r,s} =
    \begin{cases}
      1 & \mbox{if \( s=n+r \)},\\
      -\frac{n}{(r+1)(s+2)} & \mbox{if \( n-1\le s < n+r \)},\\
      0 & \mbox{otherwise.}
    \end{cases}
  \]
\end{prop}
When \( 0 < a < 1 \), the Verblunsky coefficients are given by
\[
  \alpha_n = -\frac{u-u^{-1}}{u^{n+2} - u^{-n-2}}, \quad\mbox{where} \quad u = a^{-1} + \sqrt{a^{-2}-1}.
\]
\begin{prop}\label{pro:SNM_nm_a<1}
  For the OPUC that correspond to the single nontrivial moment measure with \( 0 < a < 1 \), we have 
  \[
    \mu_{n,m} =
    \begin{cases}
      1 & \mbox{if \( m=n \)},\\
      - \frac{u^{n}-u^{-n}}{u^{n+1}-u^{-(n+1)}} & \mbox{if \( m=n-1 \)},\\
      0 & \mbox{otherwise.}
    \end{cases}
  \]
\end{prop}
\begin{proof}
  Let \( \lambda_{n,m} \) be the right-hand side of the equation in
  the proposition and let
  \( f_n = -\lambda_{n,n-1}= \frac{u^{n}-u^{-n}}{u^{n+1}-u^{-(n+1)}}
  \). It suffices to show that
  \begin{equation}\label{eq:snm}
    \lambda_{n,m} = \lambda_{n-1,m-1} - \frac{1}{f_{m+1}} \lambda_{n-1,m} +  \frac{1}{f_{m+1}}  \left( 1-\left( \frac{u-u^{-1}}{u^{m+2}-u^{-(m+2)}} \right)^2 \right) \lambda_{n,m+1},
  \end{equation}
  where the initial conditions are \( \lambda_{n,m} = \delta_{n,m} \)
  for \( n\le m \) and \( \lambda_{n,-1} = 0 \) for \( n\ge0 \), which
  is obvious. If \( m=n \), then~\eqref{eq:snm} immediately follows
  from the initial conditions. If \( m = n-1 \), then using the fact
  that
  \( \frac{1}{f_{m+1}} \left( 1-\left(
      \frac{u-u^{-1}}{u^{m+2}-u^{-(m+2)}} \right)^2 \right) =
  \frac{1}{f_{m+2}} \), we obtain
  \begin{align*}
    \lambda_{n-1,n-2} - \frac{1}{f_{n}} \lambda_{n-1,n-1} + \frac{1}{f_{n+1}} \lambda_{n,n} = -f_{n-1} - \frac{1}{f_n} + \frac{1}{f_{n+1}} = -f_{n} = \lambda_{n,n-1}.
  \end{align*}
  If \( m = n-2 \), then
  \begin{align*}
    \lambda_{n-1,n-3} -   \frac{1}{f_{n-1}}\lambda_{n-1,n-2} +  \frac{1}{f_{n}}\lambda_{n,n-1} = 0 -  \frac{1}{f_{n-1}}(-f_{n-1}) + \frac{1}{f_{n}} (-f_n) = 0 = \lambda_{n,n-2}.
  \end{align*}
  Finally, if \( m < n-2 \) or \( m > n \), then \eqref{eq:snm} also
  holds since both sides are equal to zero. This completes the proof.
\end{proof}

\begin{prop}\label{pro:SNM_nrs_a<1}
  For the OPUC that correspond to the single nontrivial moment measure
  with \( 0 < a < 1 \), we have
  \[
    \mu_{n,r,s} =
    \begin{cases}
      1 & \mbox{if \( s=n+r \)},\\
      - \frac{(u^{n}-u^{-n})(u-u^{-1})}{(u^{s+2}-u^{-(s+2)})(u^{r+1}-u^{-(r+1)})} & \mbox{if \( n-1 \le s < n+r \)},\\
      0 & \mbox{otherwise.}
    \end{cases}
  \]
\end{prop}
\begin{proof}
  It is known \cite[(1.6.31)]{Ismail2009} that
  \( \Phi_n(z) = \alpha_{n-1}\sum_{j=0}^n \frac{1}{\alpha_{j-1}}z^j =
  \sum_{j=0}^n \frac{u^{j+1}-u^{-(j+1)}}{u^{n+1}-u^{-(n+1)}}z^j \).
  By~\eqref{eq:nrs_from_nm}, we have
  \[
    \mu_{n,r,s} =  \sum_{j=0}^r \frac{u^{j+1}-u^{-(j+1)}}{u^{r+1}-u^{-(r+1)}}\mu_{n+j,s}.
  \]
  If \( s = n+r \), then
  \( \mu_{n,r,s} =
  \frac{u^{r+1}-u^{-(r+1)}}{u^{r+1}-u^{-(r+1)}}\mu_{n+r,n+r} = 1 \).
  If \( n-1 \le s <n+r \), then
  \begin{align*}
    \mu_{n,r,s} &= \frac{u^{s-n+1}-u^{-\left( s-n+1\right) }}{u^{r+1}-u^{-\left( r+1\right) }}+
                  \frac{u^{s-n+2}-u^{-\left( s-n+2\right) }}{u^{r+1}-u^{-\left( r+1\right) }}
                  \left( -\frac{u^{s+1}-u^{-\left( s+1\right) }}{u^{s+2}-u^{-\left( s+2\right)
                  }}\right) \\
                &= -\frac{\left( u^{n}-u^{-n}\right) \left( u-u^{-1}\right) }{\left(
                  u^{s+2}-u^{-\left( s+2\right) }\right) \left( u^{r+1}-u^{-\left( r+1\right)
                  }\right) }.
  \end{align*}
  Otherwise, we have \( \mu_{n,r,s} = 0 \) since \( \mu_{n+j,s} = 0 \)
  for \( 0 \le j \le r < s-n \).
\end{proof}
 
 \begin{remark}\label{rem:1}
   \Cref{pro:SNM_nm_a1} can also be obtained
   from~\Cref{pro:SNM_nrs_a<1} by taking the limit as \( a \)
   approaches to \( 1^- \).
 \end{remark}

\subsection{Rogers--Szeg\H{o} polynomials}
For \( q \in (0,1) \), the Rogers--Szeg\H{o} polynomials are defined
in \cite[Example~1.6.5]{Simon2005a} by the wrapped Gaussian measure
\( d\mu_q(\theta) = 2\pi(2\pi a)^{-1/2} \sum_{j = - \infty}^{\infty}
e^{-(\theta-2\pi j)^2/2a} \frac{d\theta}{2\pi} \), and are also
defined by the Verblunsky coefficients
\[
  \alpha_n = (-1)^n q^{(n+1)/2}, \qquad  n \geq 0.
\] 

\begin{prop}\label{pro:RogersSzego_nm}
  For the Rogers--Szeg\H{o} polynomials, we have
  \[
    \mu_{n,m} = \qbinom{n}{m} q^{\frac{(n-m)^2}{2}}, 
  \]
\end{prop}
where \( \qbinom{n}{m} \) is a \( q \)-binomial given by
\( \qbinom{n}{m}
=\frac{(1-q^n)(1-q^{n-1})\cdots(1-q^{n-m+1})}{(1-q^m)(1-q^{m-1})\cdots(1-q)}
\).
\begin{proof}
  Let \( \sigma_{n,m} \) be the right-hand side of the equation. We
  show that \( \sigma_{n,m} \) has the recurrence relation
  \begin{align*}
    \sigma_{n,m}
    &= \sigma_{n-1,m-1} - \frac{\overline{\alpha_{m-1}}}{\overline{\alpha_{m}}} \sigma_{n-1,m} +\frac{\overline{\alpha_{m-1}}}{\overline{\alpha_{m}}}(1-|\alpha_m|^2)\sigma_{n,m+1} \\
    &= \sigma_{n-1,m-1} + q^{(-\frac{1}{2})} \sigma_{n-1,m} -q^{(-\frac{1}{2})}(1-|q|^{m+1}) \sigma_{n,m+1},
  \end{align*}
  with the initial conditions \( \sigma_{n,m} = \delta_{n,m} \) for
  \( n \le m \) and \( \sigma_{n,-1} =0 \) for all \( n \geq 0 \).
  From \( \qbinom{n-1}{m-1} = \frac{1-q^m}{1-q^n} \qbinom{n}{m} \), we
  immediately obtain
  \begin{equation}\label{eq:RS_rec1}
    \sigma_{n-1,m-1} = \frac{1-q^m}{1-q^n}\sigma_{n,m}. 
  \end{equation}
  Moreover, we use
  \begin{align*}
    \qbinom{n-1}{m}q^{\binom{n-1-m}{2}}q^{\frac{n-1-m}{2}} &= \frac{1-q^{n-m}}{1-q^n}q^{-(n-1-m)}q^{(-\frac{1}{2})}\qbinom{n}{m}, \quad\mbox{and}  \\
    \qbinom{n}{m+1}q^{\binom{n-m-1}{2}}q^{\frac{n-m-1}{2}} &=  \frac{1-q^{n-m}}{1-q^{m+1}}q^{-(n-1-m)}q^{(-\frac{1}{2})}\qbinom{n}{m},
  \end{align*}
  to derive
  \begin{align}
    \label{eq:RS_rec2}q^{(-\frac{1}{2})} \sigma_{n-1,m} &= \frac{q^{m-n}-1}{1-q^n} \sigma_{n,m}, \quad\mbox{and}\\
    \label{eq:RS_rec3}q^{(-\frac{1}{2})}(1-q^{m+1}) \sigma_{n,m+1} &= (q^{m-n}-1) \sigma_{n,m}.
  \end{align}
  Therefore, we obtain the desired recurrence relation from
  \eqref{eq:RS_rec1}, \eqref{eq:RS_rec2}, \eqref{eq:RS_rec3} together
  with a direct computation
  \[
    \frac{1-q^m}{1-q^n} + \frac{q^{m-n}-1}{1-q^n} - (q^{m-n}-1) = 1.\qedhere
  \]
\end{proof}

\begin{prop}\label{pro:RogersSzego_nrs}
  For the Rogers--Szeg\H{o} polynomials, we have
 \[
   \mu_{n,r,s} = \sum_{j=0}^r (-1)^{r-j}\qbinom{r}{j} \qbinom{n+j}{s} q^{\frac{r-j}{2}}q^{\frac{(n+j-s)^2}{2}},
 \]
 where \( \qbinom{n}{m} \) is a \( q \)-binomial given by
 \( \qbinom{n}{m}
 =\frac{(1-q^n)(1-q^{n-1})\cdots(1-q^{n-m+1})}{(1-q^m)(1-q^{m-1})\cdots(1-q)}
 \).
 \end{prop}
\begin{proof}
  It is known \cite[Theorem 1.6.7]{Simon2005a} that  \(  \Phi_n(z) = \sum_{j=0}^n(-1)^{n-j}\qbinom{n}{j} q^{\frac{n-j}{2}} \). By~\eqref{eq:nrs_from_nm}, we have
  \begin{align*}
    \mu_{n,r,s}&= \sum_{j=0}^r(-1)^{r-j}\qbinom{r}{j}q^{\frac{r-j}{2}} \mu_{n+j,s}  \\
               &=  \sum_{j=0}^r(-1)^{r-j}\qbinom{r}{j} \qbinom{n+j}{s} q^{\frac{r-j}{2}}q^{\frac{(n+j-s)^2}{2}}.\qedhere
  \end{align*}
\end{proof}

\begin{remark}
  Following \cite[p.82]{Simon2005a}, one can extend
  \Cref{pro:RogersSzego_nm} and \Cref{pro:RogersSzego_nrs} for
  \( q \in (0,1) \) to the results for a complex number
  \( |q|\omega\in\DD \) with \( |\omega| = 1 \). The Verblunsky
  coefficients and \( \Phi_n(z) \) are given by
  \begin{align*}
    \alpha_n &= -(-\overline{\omega}^{n+1})|q|^{(n+1)/2}, \\
    \Phi_n(z) &=  \sum_{j=0}^n(-\omega)^{n-j}\Qbinom{n}{j}{|q|} |q|^{\frac{n-j}{2}} z^j.
  \end{align*}
  One can obtain \( \mu_{n,m} \) and \( \mu_{n,r,s} \) in a similar
  way. We leave it to the readers.
\end{remark}

We end this section with another type of an example. The
Al-Salam--Carlitz $q$-polynomials on the unit circle have the
following Verblunsky coefficients:
\[
  \alpha_{2n} = 0, \quad \alpha_{2n+1} = 1-2q^{n+1}, \quad n \geq 0.
\]
For all other examples except the Bernstein--Szeg\H{o} polynomials, we
used Schr\"oder path model to obtain \( \mu_{n,m} \) and
\( \mu_{n,r,s} \) since \( \alpha_n \ne 0 \). For the
Bernstein--Szeg\H{o} polynomials, we were able to find \( \mu_{n,m} \)
and \( \mu_{n,r,s} \) directly from the formula of \( \Phi_n(z) \).
However, none of these methods works for the Al-Salam--Carlitz
\( q \)-polynomials. It would be interesting to find explicit forms of
\( \mu_{n,m} \) and \( \mu_{n,r,s} \) for them.

\section{Generalized linearization coefficients and their variations}
\label{sec:generalized_lc}

Let us consider
\[
  z^n\Phi_r(z) = \sum_{s\geq 0}a_{n,r,s} \Phi_s(z).
\]
We call \( a_{n,r,s} \) a \emph{generalized linearization
  coefficient}, and we provided combinatorial interpretations for the
conjugate of it in \Cref{thm:Luka} and \Cref{thm:gentle Motzkin},
which states:
\[
  \mu_{n,r,s} = \overline{a_{n,r,s}}.
\]
We introduce other kinds of generalized linearization coefficients,
\( b_{n,r,s}, c_{n,r,s} \) and \( d_{n,r,s} \):
\begin{align}
  \label{eq:lc1} z^n\Phi_r^*(z) &= \sum_{s\ge0}b_{n,r,s}\Phi_s(z),\\
  \label{eq:lc2} z^n\Phi_r(z) &= \sum_{s\ge0}c_{n,r,s}\Phi_s^*(z),\\
  \label{eq:lc3} z^n\Phi_r^*(z) &= \sum_{s\ge0}d_{n,r,s}\Phi_s^*(z).
\end{align}
For brevity, we often refer to \( a_{n,r,s} \),
\( b_{n,r,s}, c_{n,r,s} \), and \( d_{n,r,s} \) as linearization
coefficients. In this section, we give combinatorial interpretations
for
\[
  \nu_{n,r,s}:= \overline{b_{n,r,s}},\quad \eta_{n,r,s}:= \overline{c_{n,r,s}} \qand \theta_{n,r,s} := \overline{d_{n,r,s}},
\]
as the second application of combinatorial interpretations for the
generalized moments. Note that if \( \alpha_i=0 \) for some
\( 0 \le i \le n+r \), then by~\eqref{eq:rec2} we have
\( \deg \Phi_s^*<s \) for \( s > i \), whereas we always have
\( \deg \Phi_s = s \) for any \( \alpha_i \). Thus, the linearization
coefficients \( c_{n,r,s} \) and \( d_{n,r,s} \) are uniquely
determined only when \( \alpha_i\ne 0 \) for \( 0\le i \le n+r \).
This is the reason why we assume that \( \alpha_i\ne 0 \) for
\( i\geq 0 \) in~\Cref{pro:phi_star_c_nrs}
and~\Cref{pro:Phi_star_d_nrs}.

For \eqref{eq:lc1}, the conjugate of the linearization coefficient
\( b_{n,r,s} \) can be written as
\begin{align*}
  \overline{b_{n,r,s}} = \frac{\langle \Phi_s,z^n\Phi_r^* \rangle}{\langle \Phi_s,\Phi_s \rangle}.
\end{align*}

We give a combinatorial interpretation for
\( \nu_{n,r,s}=\overline{b_{n,r,s}} \).

\begin{prop}\label{prop:phi_star1}
  For nonnegative integers \( n \), \( r \), and \( s \), we have
  \[
    \nu_{n,r,s} =\alpha_r^{-1} \sum_{p}\wt_L(p),
  \]
  where the summation is over all \L{}ukasiewicz paths from
  \( (-1,r) \) to \( (n,s) \) that do not start with an up-step
  \( (1,1) \). The weight function \( \wt_L \) is defined in
  \Cref{def:Luka}.
\end{prop}
\begin{proof}
  Using \eqref{eq:rec2}, one can see that
  \[
    \nu_{n,r+1,s} = \nu_{n,r,s}-\overline{\alpha_r}\mu_{n+1,r,s}.
  \]
  Since \( \mu_{n,r,s} \) for \( n,r,s \geq 0 \) are given, the value
  \( \nu_{n,r,s} \) is uniquely determined by the above recurrence
  relation with the initial condition \( \nu_{n,0,s} = \mu_{n,0,s} \).
  The initial condition follows by definition. Let
  \( \widehat{\nu}_{n,r,s} \) be the right-hand side of the equation
  in the statement. It suffices to show that
  \begin{equation}\label{eq:nu_rec}
   \widehat{\nu}_{n,r+1,s} = \widehat{\nu}_{n,r,s}-\overline{\alpha_r}\mu_{n+1,r,s},
  \end{equation}
  and the initial condition \( \widehat{\nu}_{n,0,s} = \mu_{n,0,s} \).
  This can be seen by examining \( \widehat{\nu}_{n,r+1,s} \)
  according to the first step as follows:
  \begin{align*}
    \widehat{\nu}_{n,r+1,s} =& \sum_{k=1}^{r+1}\left(-\overline{\alpha_{r-k}}\prod_{j=r+1-k}^{r}(1-|\alpha_j|^2)\right)\mu_{n,r+1-k,s} - \overline{\alpha_{r}}\mu_{n,r+1,s} \\
                             =&  \sum_{k=1}^{r+1}\left(-\overline{\alpha_{r-k}}\prod_{j=r+1-k}^{r-1}(1-|\alpha_j|^2)\right)\mu_{n,r+1-k,s} \\
                             &- \sum_{k=1}^{r+1}\overline{\alpha_{r}}\left(-\alpha_r\overline{\alpha_{r-k}}\prod_{j=r+1-k}^{r-1}(1-|\alpha_j|^2)\right)\mu_{n,r+1-k,s} - \overline{\alpha_{r}}\mu_{n,r+1,s} \\
                             =&  \sum_{k=0}^{r}\left(-\overline{\alpha_{r-k-1}}\prod_{j=r-k}^{r-1}(1-|\alpha_j|^2)\right)\mu_{n,r-k,s} - \overline{\alpha_{r}}\mu_{n+1,r,s} \\
                             =& \widehat{\nu}_{n,r,s}-\overline{\alpha_r}\mu_{n+1,r,s}.
  \end{align*}
  For the initial condition, since \( \wtL((-1,0) \to (0,0)) = \alpha_0 \), we immediately have \( \widehat{\nu}_{n,0,s} = \mu_{n,0,s} \) by~\Cref{thm:Luka}.
 \end{proof}

 \begin{remark}\label{rem:2}
   As we discussed in ~\Cref{rmk:RI_OPUC}, we have
   \[
     \sum_p \wtS(p) =  \frac{1}{\overline{\alpha_s}}\cdot\frac{\langle z^{n+1}\Phi_r,\Phi_s^* \rangle}{\langle \Phi_s,\Phi_s \rangle},
   \]
   where the sum is over all Schr\"oder paths from \( (0,r) \) to
   \( (n,s) \). The right-hand side of the above equation is written
   in terms of
   \( \nu_{n,r,s} = \frac{\langle \Phi_s,z^n\Phi_r^* \rangle}{\langle
     \Phi_s,\Phi_s \rangle} \):
   \begin{align*}
     \frac{1}{\overline{\alpha_s}}\cdot\frac{\langle z^{n+1}\Phi_r,\Phi_s^* \rangle}{\langle \Phi_s,\Phi_s \rangle} &=  \frac{1}{\overline{\alpha_s}}\cdot \frac{\langle \Phi_r, z^{-n-1}\Phi_s^* \rangle}{\langle \Phi_s, \Phi_s \rangle}  = \frac{1}{\overline{\alpha_s}}\cdot\frac{\langle \Phi_r, \Phi_r \rangle}{\langle \Phi_s,\Phi_s \rangle}\cdot\nu_{-n-1,s,r}. 
   \end{align*}
  This leads to an immediate combinatorial interpretation for \( \nu_{-n,r,s} \):
   \[
     \nu_{-n,r,s} = \overline{\alpha_r} \cdot \frac{\prod_{j=0}^{r-1}(1-|\alpha_j|^2)}{\prod_{j=0}^{s-1}(1-|\alpha_j|^2)} \cdot  \sum_p \wtS(p), 
   \]
   where the sum is over all Schr\"oder paths from \( (0,s) \) to \( (n-1,r) \).
   
 \end{remark}

 Next, for \eqref{eq:lc2}, we cannot apply a similar argument since
 \( \langle \Phi_k^*,\Phi_s^* \rangle \) and
 \( \langle \Phi_k,\Phi_s^* \rangle \) are not zero even for
 \( k < s \). We give a combinatorial description for
 \( \eta_{n,r,s}=\overline{c_{n,r,s}} \) in a different way using
 matrix inverses. To do this, we need an additional condition
 \( \alpha_i\ne0 \) for \( i\geq 0 \).

\begin{prop}\label{pro:phi_star_c_nrs}
  For nonnegative integers \( n \), \( r \), and \( s \), let
  \( \alpha_i\ne0 \) for \( 0 \le i \le n+r \). Then, for \( n\ge1 \),
  we have
  \[
    \eta_{n,r,s}=-\left(\overline{\alpha_{s-1}}\right)^{-1}\sum_{p}\wt_S(p),
  \]
  where the summation is over all Schr\"oder paths in
  \( \Sch_{n,r,s} \) that do not also end with a vertical down-step
  \( (0,-1) \), and where the weight function \( \wt_S \) is defined
  in \Cref{def:Sch}. For \( n = 0 \), the value \( \eta_{0,r,s} \) is
  given by
  \[
    \eta_{0,r,s}=\begin{cases}
                   -\frac{1}{\overline{\alpha_{r-1}}} & \mbox{if } s=r,\\
                   \frac{1-|\alpha_{r-1}|^2}{\overline{\alpha_{r-1}}} & \mbox{if } s=r-1, \\
                   0 & \mbox{otherwise.}
                 \end{cases}
               \]
             \end{prop}
\begin{proof}
  Taking \( \langle \Phi_k,\cdot \rangle \) and dividing by
  \( \langle \Phi_k,\Phi_k \rangle \) on both sides of~\eqref{eq:lc2},
  we obtain
  \[
    \mu_{n,r,k}=\sum_{s \geq 0}\eta_{n,r,s}\nu_{0,s,k}.
  \]
  This can be written as a matrix equation:
  \[
    (\mu_{n,i,j})_{i,j\ge0} (\nu_{0,i,j})_{i,j\ge0}^{-1} = (\eta_{n,i,j})_{i,j\ge0}.
  \]
  Using \Cref{prop:phi_star1}, we can compute the inverse matrix
  \( (\nu_{0,i,j})_{i,j\ge0}^{-1}=(\tau_{i,j})_{i,j\ge0} \), where
  \[
    \tau_{i,j} =
    \begin{cases}
      -\frac{1}{\overline{\alpha_{j-1}}}& \mbox{if } i=j,\\
      \frac{1-|\alpha_j|^2}{\overline{\alpha_j}}& \mbox{if } i=j+1,\\
      0 & \mbox{otherwise.}
    \end{cases}
  \]
  Thus, we have
  \begin{equation}\label{eq:c_nrs}
    \eta_{n,r,s} = \sum_{k \geq 0}\mu_{n,r,k}\tau_{k,s}=\frac{\mu_{n,r,s+1}(1-|\alpha_s|^2)}{\overline{\alpha_s}}-\frac{\mu_{n,r,s}}{\overline{\alpha_{s-1}}} = -\frac{1}{\overline{\alpha_{s-1}}}\left(\mu_{n,r,s}-\frac{\overline{\alpha_{s-1}}}{\overline{\alpha_s}}(1-|\alpha_s|^2)\mu_{n,r,s+1}\right).
  \end{equation}
  If we interprete \( \mu_{n,r,s} \) using the step set
  in~\Cref{def:Sch}, we see that the term in the parenthesis on the
  rightmost side means that we do not use a vertical step at the end.
  This completes the proof.
\end{proof}

Finally, for \( \theta_{n,r,s}=\overline{d_{n,r,s}} \), we can apply a
similar argument as in the proof of \Cref{pro:phi_star_c_nrs}.

\begin{prop}\label{pro:Phi_star_d_nrs}
  For nonnegative integers \( n \), \( r \) and \( s \), let
  \( \alpha_i\ne0 \) for \( 0\le i\le n+r \). For \( n\geq 1 \), we
  have
   \[
     \theta_{n,r,s}=\frac{\overline{\alpha_{r-1}}}{ \overline{\alpha_{s-1}}\cdot |\alpha_r|^2}\sum_{p}\wt_S(p),
  \]
  where the summation is over all Schr\"oder paths from \( (0,r) \) to
  \( (n,s) \) that do not end with a vertical down-step \( (0,-1) \),
  and where the weight function \( \wt_S \) is defined in
  \Cref{def:Sch}. For \( n=0 \), the value \( \theta_{0,r,s} \) is
  given by \( \theta_{0,r,s}=\delta_{r,s}. \)
\end{prop}
\begin{proof}
  The proposition can be proved in a similar way to the proof of
\Cref{pro:phi_star_c_nrs}. We omit the detail, but the key is to
obtain
  \begin{equation}\label{eq:d_nrs} \theta_{n,r,s} =
-\frac{1}{\overline{\alpha_{s-1}}}\left(\nu_{n,r,s}-\frac{\overline{\alpha_{s-1}}}{\overline{\alpha_s}}(1-|\alpha_s|^2)\nu_{n,r,s+1}\right)
  \end{equation} by rewriting~\eqref{eq:lc3} in terms of a matrix
equation. Using the grouping idea introduced in the proof of
\Cref{prop:Luka=Sch}, we obtain \( \nu_{n,r,s} = \alpha_r^{-1}\sum_p
\wt_S(p) \), where the sum is over all Schr\"oder paths from \( (-1,r)
\) to \( (n,s) \) that must start with the horizontal step of weight
\( -\overline{\alpha_{r-1}}\cdot(\overline{\alpha_{r}})^{-1} \). This
weight comes from the fact that the \L{}ukasiewicz paths used for \(
\nu_{n,r,s} \) do not start with an up-step. All other steps have the
weight in~\Cref{def:Sch}. Extracting the weight of the step \(
(-1,r)\to (0,r) \), one can obtain the result.
\end{proof}

\begin{remark}
The value \( \nu_{n,r,s} \) is a polynomial in variables \(
\alpha_0,\overline{\alpha_0}, \alpha_1, \overline{\alpha_1}, \dots \),
and it is \( \{\alpha_0,-\overline{\alpha_0}, \alpha_1,
-\overline{\alpha_1}, \dots \} \)-positive. However,
by~\Cref{pro:phi_star_c_nrs} and~\Cref{pro:Phi_star_d_nrs}, we see
that two linearization coefficients \( \eta_{n,r,s} \) and \(
\theta_{n,r,s} \) are not polynomials in variables \(
\alpha_0,\overline{\alpha_0}, \alpha_1, \overline{\alpha_1}, \dots \).
Furthermore, even though \(
\overline{\alpha_{s}}\overline{\alpha_{s-1}}\eta_{n,r,s} \) and \(
\overline{\alpha_{s-1}}|\alpha_r|^2\theta_{n,r,s} \) are polynomials,
they are not \( \{\alpha_0,-\overline{\alpha_0}, \alpha_1,
-\overline{\alpha_1}, \dots \} \)-positive.
\end{remark}
\begin{remark}
  We can rewrite \( \eta_{n,r,s} \) using the grouping idea in the
proof of \Cref{prop:Luka=Sch} as
   \[ \eta_{n,r,s}=(\overline{\alpha_{s}})^{-1}\sum_{p\in
\Luka_{n-1,r,s}}\wt_L(p)-(\overline{\alpha_{s-1}})^{-1}\sum_{p\in
\Luka_{n-1,r,s-1}}\wt_L(p),
  \] where the weight function \( \wt_L \) is defined in
\Cref{def:Luka}. The value \( \theta_{n,r,s} \) can be given in a
similar way. We can also check that \(
\overline{\alpha_{s}}\overline{\alpha_{s-1}}\eta_{n,r,s} \) and \(
\overline{\alpha_{s-1}}|\alpha_r|^2\theta_{n,r,s} \) are polynomials
using this idea.
\end{remark}

\section*{Acknowledgement}
The authors thank Jang Soo Kim for suggesting the problem that
inspired this work and for his careful proofreading and helpful
suggestions. We also thank Jisun Huh, Byung-Hak Hwang and Jaeseong Oh
for their careful review of the final draft, and an anonymous referee
for valuable comments on the manuscript. The authors were supported by
the National Research Foundation of Korea(NRF) grants
\#2022R1A2C101100911. J. Jang was supported by the BK21 FOUR Project.
M. Song was supported by the National Research Foundation of
Korea(NRF) grant funded by the Korea government(MSIT) (No.
2022R1C1C2009025).

\bibliographystyle{plain}

\end{document}